\documentclass[12pt]{article}
\usepackage[utf8]{inputenc}
\usepackage[english]{babel}
\usepackage{amssymb,amsmath,graphicx}
\usepackage{amsthm}
\usepackage{float}
\usepackage{afterpage}
\usepackage{dcolumn}
\usepackage[round,authoryear]{natbib}
\usepackage{enumitem}
 \usepackage{setspace}
\usepackage{cases}
\usepackage{bbm}
\usepackage{microtype}
\usepackage{multirow}
\usepackage{color}
\usepackage[ruled,linesnumbered,noend]{algorithm2e}
\SetArgSty{textrm} 
\usepackage{natbib}
 \bibpunct[, ]{(}{)}{,}{a}{}{,}%
 \def\BIBand{et}%
\usepackage[hidelinks]{hyperref}
\usepackage{siunitx} 
\usepackage{framed}
\usepackage{bbm} 
\allowdisplaybreaks

\definecolor{lightgray}{gray}{0.70}
\renewenvironment{leftbar}[1][\hsize]
{%
\MakeFramed{\hsize#1\advance\hsize-\width\FrameRestore}%
}
{\endMakeFramed}
\newtheorem{theorem}{Theorem}
\newtheorem{lemma}{Lemma}

\newcommand{\cO}{{\mathcal{O}}}

\title{Separable Convex Optimization with Nested Lower and Upper Constraints}

\author{Thibaut Vidal, Daniel Gribel, Patrick Jaillet}

\begin{document}

\begin{center}

\begin{LARGE}
{Separable Convex Optimization with Nested Lower \vspace*{0.15cm}\linebreak and  Upper Constraints}
\end{LARGE}

\vspace*{0.85cm}

\textbf{Thibaut Vidal*, Daniel Gribel} \\
Departamento de Informática, \\ 
Pontifícia Universidade Católica do Rio de Janeiro (PUC-Rio) \\
vidalt@inf.puc-rio.br \\
\vspace*{0.2cm}
\textbf{Patrick Jaillet} \\
Department of Electrical Engineering and Computer Science, \\
Laboratory for Information and Decision Systems,  \\
Operations Research Center, Massachusetts Institute of Technology \\
jaillet@mit.edu \\

\vspace*{0.4cm}

August 2018 

Author Accepted Manuscript (AAM)

Accepted for publication in \emph{INFORMS Journal on Optimization}

\vspace*{0.15cm}

\end{center}
\noindent
\textbf{Abstract.}
We study a convex resource allocation problem in which lower and upper bounds are imposed on partial sums of allocations. This model
is linked to a large range of applications, including production planning, speed optimization, stratified sampling, support vector machines, portfolio management, and telecommunications.
We propose an efficient gradient-free divide-and-conquer algorithm, which uses monotonicity arguments to generate valid bounds from the recursive calls, and eliminate linking constraints based on the information from sub-problems. This algorithm does not need strict convexity or differentiability. It produces an $\epsilon$-approximate solution for the continuous problem in $\cO(n \log m \log \frac{n B}{\epsilon})$ time and an integer solution in $\cO(n \log m \log B)$ time, where $n$ is the number of decision variables, $m$ is the number of constraints, and $B$ is the resource bound.
A complexity of $\cO(n \log m)$ is also achieved for the linear and quadratic cases. These are the best complexities known to date for this important problem class. Our experimental analyses confirm the good performance of the method, which produces optimal solutions for problems with up to 1,000,000 variables in a few seconds. Promising applications to the support vector ordinal regression problem are also investigated.
\vspace*{0.3cm}


\vspace*{0.5cm}

\noindent
* Corresponding author

\newpage

\section{Introduction}
\label{Introduction}

Resource allocation problems involve the distribution of a fixed quantity of a resource (e.g., time, workforce, money, energy) over a number of tasks in order to optimize a value function. In its most fundamental form, the simple resource allocation problem (RAP) is formulated as the minimization of a separable objective subject to one linear constraint representing the total resource bound. Despite its apparent simplicity, this model has been the focus of a considerable research effort over the years, with more than a hundred articles, as underlined by the surveys of \cite{Patriksson2008}, \cite{Katoh2013}, and \cite{Patriksson2015}. This level of interest arises from its applications in engineering, production and manufacturing, military operations, machine learning, financial economics, and telecommunications, among many other areas.

In several applications, a single global resource bound is not sufficient to model partial budget or investment limits, release dates and deadlines, or inventory or workforce limitations. In these situations, the problem must be generalized to include additional constraints over nested sums of the resource variables. This often leads to the model given in Equations (\ref{LBUB1:1})--(\ref{LBUB1:3}), where the sets~$J_i \subseteq \{1,\dots,n\}$ follow a total order such that  $J_i \subset J_{i+1}$ for $i \in \{1,\dots,m-2\}$:  \vspace*{-0.3cm}
\begin{align}
\min  \hspace*{0.1cm}  f(\mathbf{x}) =  &\sum\limits_{i=1}^{n}  f_i(x_i) & \label{LBUB1:1} \\
\text{s.t.}  \hspace*{0.3cm}  a_i \leq  & \sum\limits_{j \in J_i}  x_j  \leq  \  b_i   &i  \in \{ 1,\dots,m-1\} \label{LBUB1:2} \\
&  \sum\limits_{k=1}^{n}  x_k   =  B   &  \label{LBUB1:2bis} \\
  c_i \leq & \ x_i \leq  d_i & i  \in \{ 1,\dots,n\}. \label{LBUB1:3}
\end{align}

This problem involves a separable convex objective, subject to lower and upper bounds on nested subsets of the variables (Equation \ref{LBUB1:2}) and a global resource constraint (Equation~\ref{LBUB1:2bis}).
Despite being a special case of the former inequalities, the latter constraint is included in the model to emphasize the resource bound $B$. Re-ordering the indices of $x_i$, $c_i$, $d_i$ and $f_i$, we obtain the formulation given by \mbox{Equations  (\ref{LBUB2:1})--(\ref{LBUB2:3})}, where $(\sigma[1],\dots,\sigma[m])$ is a subsequence of $(1,\dots,n)$:  \vspace*{-0.3cm}
\begin{align}
\min  \hspace*{0.1cm}  f(\mathbf{x}) =  &\sum\limits_{i=1}^{n}  f_i(x_i) & \label{LBUB2:1} \\
\text{s.t.}  \hspace*{0.3cm}  a_i \leq  &  \sum\limits_{k=1}^{\sigma[i]}  x_k   \leq  \  b_i   &i  \in \{ 1,\dots,m-1\} \label{LBUB2:2} \\
&  \sum\limits_{k=1}^{n}  x_k   =  B   &  \label{LBUB2:2bis} \\
  c_i \leq & \ x_i \leq  d_i & i  \in \{ 1,\dots,n\}. \label{LBUB2:3}
\end{align}
We assume that the functions $f_i : [c_i, d_i] \rightarrow \Re$ are Lipschitz continuous but not necessarily differentiable or strictly convex, and the coefficients $a_i$, $b_i$, $c_i$, and $d_i$ are integers. To ease the presentation, we define $a_m = b_m = B$, $\sigma[0] = 0$ and $\sigma[m] = n$. We will study this continuous optimization problem as well as its restriction to integer solutions.

We refer to this problem as the RAP with nested lower and upper constraints \mbox{(RAP--NC)}.
As highlighted in Section~\ref{Applications}, the applications of this model include production and capacity planning \citep{Love1973}, vessel speed optimization \citep{Psaraftis2013,Psaraftis2014}, machine learning \citep{Chu2007a}, portfolio management \citep{Bienstock1996}, telecommunications \citep{DAmico2014a} and power management \citep{Gerards2016}. Some of these applications involve large data sets with millions of variables, and in other contexts multiple RAP--NC must be repeatedly solved (e.g., thousands or millions of times) to produce bounds in a tree-search-based algorithm, to optimize vessel speeds over candidate routes within a heuristic search for ship routing, or to perform projection steps in a subgradient procedure for a nonseparable objective. In these situations, complexity improvements are a determining factor between success and failure.

The literature contains a rich set of studies and algorithms for a closely related problem: the NESTED resource allocation problem, a special case of the RAP--NC in which $a_i = -\infty$ for all $i \in \{1,\dots,m-1\}$ (or $b_i = \infty$ for all $i \in \{1,\dots,m-1\}$). With integer variables, NESTED can be solved in $\cO(n \log n \log \frac{B}{n})$ time using a scaling algorithm \citep{Hochbaum1994}, and $\cO(n \log m \log \frac{B}{n})$ time using divide-and-conquer principles \citep{Vidal2014a}. These algorithms, however, are not applicable for joint lower and upper nested constraints, a case which is essential for a large variety of applications, e.g., to model time windows, time-dependent inventory bounds, or investment ranges.

Moreover, in the presence of continuous variables, the notion of computational complexity for convex problems must be carefully defined, since optimal solutions can be irrational and thus not representable in a bit-size computational model. We will use the same conventions as \mbox{\cite{Hochbaum1994}}, by measuring the computational complexity of achieving an \mbox{\emph{$\epsilon$-approximate}} solution, guaranteed to be located in the solution space no further than $\epsilon$ from an optimal solution. We also assume that an oracle is available to evaluate each function $f_i$ in $\cO(1)$ time. When considering such a model of computation, controlling algorithmic approximations can be a hard task. To circumvent this issue, we will use, as in  \cite{Hochbaum1994}, a proximity theorem to transform a continuous problem into an integer problem scaled by an appropriate factor, and to translate the integer solution back to a continuous solution with the desired precision.

\noindent
The main contributions of this paper are the following:
\begin{itemize}
\item We propose an efficient decomposition algorithm for the convex RAP--NC with integer variables.
This algorithm is a non-trivial generalization of that of \cite{Vidal2014a}, and attains the same complexity of $\cO(n \log m \log B)$. Based on a proximity theorem from \mbox{\cite{Moriguchi2011}}, we extend this algorithm to solve the continuous problem in $\cO(n \log m \log \frac{n B}{\epsilon})$ time.
These are the best known complexities, to date, for both problem variants.
The complexity depends on the magnitude of $\log (B/\epsilon)$, a dependency which is known to be unavoidable in the arithmetic model of computation for general forms of convex functions \mbox{\citep{Renegar1987}}. Moreover, the algorithm calls only the oracle for the objective function, without need of gradient information, and does not rely on strict convexity or differentiability.
Finally, Lipschitz continuity is assumed for convenience in the proofs but is not mandatory, as an alternative proof line based on submodular optimization could be adopted otherwise.

\item For the specific case of quadratic functions, with continuous or integer variables, the method runs in $\cO (n \log m)$ time, hence extending the short list of quadratic problems known to be solvable in strongly polynomial~time. This also resolves an open question from \cite{Moriguchi2011}: \emph{``It is an open question whether there exist $\cO (n \log n)$ algorithms for (Nest) with quadratic objective functions''}.

\item We present computational experiments that demonstrate the good performance of the method. We compare it with a known algorithm for the linear case and a general-purpose separable convex optimization solver for the convex case, using benchmark instances derived from three families of applications.

\item We finally integrate the proposed algorithm as a projection step in a projected gradient algorithm for the support vector ordinal regression problem, highlighting promising connections with the machine learning literature.
\end{itemize}

The rest of the paper is organized as follows. 
In Section \ref{Applications}, we review related works in a wide variety of application domains. In Section \ref{section:method}, we describe the proposed algorithm and prove its correctness. In Section \ref{section experiments}, we report our computational experiments, considering linear and convex objectives, as well as support vector ordinal regression problems. Finally, in Section \ref{sec:conclusions}, we provide some concluding remarks.

\section{Related Literature and Applications}
\label{Applications}

We now review the many applications of the RAP--NC, starting with classical operations research and management science applications and then moving to statistics, machine learning, and telecommunications.\\ \vspace*{-0.5cm}

\noindent
\textbf{Resource Allocation.}
The resource allocation problem (Equations \ref{LBUB1:1}, \ref{LBUB1:2bis}, and \ref{LBUB1:3}) has long been studied as a prototypical problem. The fastest known algorithms \citep{Frederickson1982,Hochbaum1994} reach a complexity of $\cO(n \log \frac{B}{n})$ for the integer problem, and can be extended to find an $\epsilon$-approximate solution of the continuous problem in $\cO(n \log \frac{B}{\epsilon})$ operations. This complexity is known to be optimal in the algebraic tree model \citep{Hochbaum1994}. Other algorithms have been developed for several generalizations of the RAP in which the constraint set forms a polymatroid. In this context, the greedy algorithm is optimal \citep{Federgruen1986}, albeit only pseudo-polynomial, and efficient scaling algorithms can be developed \citep{Hochbaum1994,Moriguchi2011}. The special case of the integer RAP--NC with $a_i = -\infty$, called NESTED problem, can be solved in $\cO(n \log n \log \frac{B}{n})$ time using a scaling algorithm \citep{Hochbaum1994} or in $\cO(n \log m \log \frac{B}{n})$ time using divide-and-conquer principles \citep{Vidal2014a}.\\ \vspace*{-0.2cm}

\noindent
\textbf{Production Planning.}
The formulation given by Equations (\ref{LBUB2:1})--(\ref{LBUB2:3}) is also encountered in early literature on production planning over time with inventory and production costs  \citep{Wagner1958}. 
One of the models most closely related to our work is that of \cite{Love1973}, with time-dependent inventory bounds. The general problem with concave costs (economies of scale) and production capacities is known to be NP-hard. The linear or convex model remains polynomial but is more limited in terms of applicability, although convex production costs can occur in the presence of a limited workforce with possible overtime.
Two relatively recent articles have proposed polynomial algorithms for the linear problem with time-dependent inventory bounds \citep{Sedeno-Noda2004,Ahuja2008}. With upper bounds $x^\textsc{max}_i$ on the production quantities, and time-dependent inventory capacities $I^\textsc{max}_i$, the problem can be stated as:
\begin{align}
\min \hspace*{0.1cm}  f(\mathbf{x},\mathbf{I}) =  &\sum\limits_{i=1}^{n}  p_i(x_i) +  \sum\limits_{i=1}^{n} \alpha_i I_i   & \label{LotSizing:1} \\
\text{s.t.} \hspace{0.3cm} I_i &= I_{i-1} + x_i - d_i & i  \in \{ 2,\dots,n\} \label{LotSizing:2}  \\
I_0 &= K  \label{LotSizing:2bis} \\
 0 &\leq I_i  \leq  I^\textsc{max}_i  & i  \in \{ 1,\dots,n\} \label{LotSizing:3} \\
 0 &\leq x_i \leq  x^\textsc{max}_i & i  \in \{ 1,\dots,n\}. \label{LotSizing:4}
\end{align}
Then, expressing the inventory variables as a function of the production quantities, using $I_i = K + \sum_{k=1}^{i} (x_k - d_k)$, reduces this problem to an \mbox{RAP--NC}:
\begin{align}
\min  \hspace*{0.1cm}  f(\mathbf{x}) =  &\sum\limits_{i=1}^{n}  p_i(x_i) + \sum\limits_{i=1}^{n} \alpha_i \left[ K + \sum_{k=1}^{i} (x_k - d_k) \right]   & \label{LotSizing2:1} \\
 \text{s.t.} \hspace{0.3cm} &\sum_{k=1}^{i} d_k  - K \leq \sum_{k=1}^{i} x_k \leq  \sum_{k=1}^{i} d_k +  I^\textsc{max}_i - K   & i  \in \{ 1,\dots,n\} \label{LotSizing2:2} \\
 &0 \leq x_i \leq  x^\textsc{max}_i & i  \in \{ 1,\dots,n\}. \label{LotSizing2:3}
\end{align}

The objective includes production costs and inventory costs, and the nested constraints model the time-dependent inventory limit.
The algorithm of \cite{Ahuja2008} can solve Equations (\ref{LotSizing:1})--(\ref{LotSizing:4}) in $\cO(n \log n)$ time via a reduction to a minimum-cost network flow problem. The method was extended to deal with possible backorders. However, this good complexity comes at the price of an advanced \emph{dynamic tree} data structure \citep{Tarjan2009} that is used to keep track of the inventory capacities.\\ \vspace*{-0.2cm}

\noindent
\textbf{Workforce Planning.}
In contrast with the above studies, which involve the production quantities as decision variables, \cite{Bellman1954} study the balancing of workforce capacity (human or technical resources) over a time horizon under hard production constraints. The variable $x_i$ now represents the workforce variation at period~$i$, and the nested constraints impose bounds on the minimum and maximum workforce in certain periods, e.g., to satisfy forecast production demand. The overall objective, to be minimized, is a convex separable cost function representing positive costs for positive or negative variations of the workforce.\\ \vspace*{-0.2cm}

\noindent
\textbf{Vessel Speed Optimization.}
In an effort to reduce fuel consumption and emissions, shipping companies have adopted \emph{slow-steaming} practices, which moderate ship speeds to reduce costs.
This line of research has led to several recent contributions on ship speed optimization, aiming to optimize the vessel speed $v_{i-1,i}$ over each trip segment of length $\delta_{i-1,i}$ while respecting a time-window constraint $[a_i,b_i]$ at each destination $i$. Let $f_i(v_{i-1,i})$ be convex functions representing the fuel costs per mile, over $(i-1,i)$, as a function of $v_{i-1,i}$, and let $t_i$ be the arrival time at $i$. The overall speed optimization problem can be formulated as: 
\begin{align}
\min \hspace{0.2cm}  &f(\mathbf{t},\mathbf{v}) = \sum\limits_{i=2}^{n}  \delta_{i-1,i} \ f_i(v_{i-1,i}) & \label{article2:Fager:1} \\
\text{s.t.}  \hspace{0.3cm}  & a_i \leq t_i \leq b_i  & i \in \{1,\dots,n\} \ \label{article2:Fager:2} \\
&  t_{i-1} + \frac{\delta_{i-1,i}}{v_{i-1,i}} \leq t_{i}  & i \in \{2,\dots,n\} \  \label{article2:Fager:3} \\
& v_{min} \leq v_{i-1,i} \leq v_{max}  & i \in \{2,\dots,n\}.    \label{article2:Fager:4}
\end{align}

Recent work has considered a constant fuel-speed trade-off function on each leg, i.e., $f_i = f_j$ for all $(i,j)$. An $\cO(n^2)$ recursive smoothing algorithm (RSA) was proposed by \cite{Norstad2010} and  \cite{Hvattum2013} for this case. However, assuming constant fuel-speed over the complete trip is unrealistic, since fuel consumption depends on many varying factors, such as sea condition, weather, current, water depth, and ship load \citep{Psaraftis2013,Psaraftis2014}. The model can be improved by dividing the trip into smaller segments and considering different functions $f_i \neq f_j$. This more general model falls outside the scope of applicability of RSA.

Let $v_i^\textsc{opt}$ be the minimum of each function $f_i$. With the change of variables $x_1 = t_1$ and $x_i = t_{i} - t_{i-1}$ for $i \geq 2$, the model can then be reformulated as:
\begin{align}
\min  \hspace{0.2cm}  & f(\mathbf{x}) = \sum\limits_{i=2}^{n}   \delta_{i-1,i} g_i\left(\frac{\delta_{i-1,i}}{x_i} \right) & \label{article2:Fagerbis:1} \\
\text{s.t.}  \hspace{0.2cm}  & a_i \leq \sum_{k=1}^{i} x_k \leq b_i  & i \in \{1,\dots,n\} \ \label{article2:Fagerbis:2} \\
&  \frac{\delta_{i-1,i}}{v_{max}} \leq x_i &  i \in \{2,\dots,n\},  \label{article2:Fagerbis:3}
\end{align}
\begin{equation}
\text{with}  \hspace{0.4cm} g_i(v) =
\begin{cases}
f_i( v_i^\textsc{opt}  ) &  \text{if} \ v \leq v_i^\textsc{opt}  \\
f_i(v) & \text{otherwise.}
\end{cases} \label{article2:Fagerbis:4}
\end{equation}

This model is a RAP--NC with separable convex cost. An efficient algorithm for this problem is critical, since a speed-optimization algorithm is not often used as a stand-alone tool but rather as a subprocedure in a route-planning algorithm \citep{Psaraftis2014}. This subprocedure can be called several million times when embedded in a local search, on subproblems counting a few hundred variables due to the division of each trip into smaller segments with different sea conditions. Finally, the RAP--NC is also appropriate for variants of vehicle routing problems with emission control \citep{Bektas2011a,Kramer2015a,Kramer2015} as well as a special case of project crashing for a known critical path \mbox{\citep{Foldes1993}}.\\ \vspace*{-0.2cm}

\noindent
\textbf{Stratified Sampling.}
Consider a population of $N$ units divided into subpopulations (\emph{strata}) of $N_1,\dots,N_n$ units such that $N_1 + \dots + N_n = N$. 
An optimized stratified sampling method aims to determine the sample size $x_i \in [0,N_i]$ for each stratum, in order to estimate a characteristic of the population while ensuring a maximum variance level $V$ and minimizing the total sampling cost.
Each subpopulation may have a different variance $\sigma_i$, so a sampling plan that is proportional to the size of the subpopulations is frequently suboptimal. The following mathematical model for this sampling design problem was proposed by \cite{Neyman1934} and extended by \cite{Srikantan1963,Hartley1965,Huddleston1970,Sanathanan1971}, and others:
\begin{align}
\min  \hspace*{0.3cm} &\sum\limits_{i=1}^{n}  c_i x_i & \label{Sampling:1} \\
\text{s.t.}  \hspace*{0.3cm} & \sum_{i=1}^n  \frac{N^2_i  \sigma^2_i}{N^2}\left(\frac{1}{x_i} -  \frac{1}{N_i} \right)  \leq V \label{Sampling:2} \\
  0 \leq & \ x_i \leq  N_i & i  \in \{ 1,\dots,n\}. \label{Sampling:3}
\end{align}

This is a classical RAP formulation. Two extensions of this model are noteworthy in our context. \cite{Hartley1965} and \cite{Huddleston1970} considered multipurpose stratified sampling where more than one characteristic is evaluated while ensuring variance bounds. This leads to several constraints of type (\ref{Sampling:2}), and thus to a continuous multidimensional knapsack problem. 
\cite{Sanathanan1971} considered a hierarchy of strata, with variance bounds for the estimates at each level. This situation occurs for example in survey sampling, when one seeks an estimate of a characteristic at both the national level (first-stage stratum) and the regional level (second-stage stratum). When two stages are considered, we obtain the additional constraints:
\begin{equation}
\sum_{i \in S_i}  \frac{N^2_i  \sigma^2_i}{N^2}\left(\frac{1}{x_i} -  \frac{1}{N_i} \right)  \leq V_i, \hspace*{0.7cm}  i \in \{1,\dots,m\}, \label{Sampling:bis} 
\end{equation}
where the $S_i$ are disjoint sets of strata, i.e., $\bigcup_{i=1}^m S_i = \{1,\dots,n\}$ and $S_i \cap S_j = \varnothing$ for all~$i,j$. The inequalities (\ref{Sampling:bis}) lead to constraints on the disjoint subsets, giving a resource allocation problem with generalized upper bounds \mbox{\citep[GUB
 --][]{Hochbaum1994,Katoh2013}}.\\ \vspace*{-0.2cm}

\noindent
\textbf{Machine Learning.}
The support vector machine (SVM) is a supervised learning model which, in its most classical form, seeks to separate a set of samples into two classes according to their labels. This problem is modeled as the search for a separating hyperplane between the projection of the two sample classes into a kernel space of higher dimension, in such a way that the classes are divided by a gap that is as wide as possible, and a penalty for misclassified samples is minimized \citep{Cortes1995}.

As a generalization of the SVM, the support vector ordinal regression (SVOR) aims to find $r-1$ parallel hyperplanes so as to separate $r$ ordered classes of samples. As reviewed in \cite{Gutierrez2016}, various models and algorithms have been proposed in recent years to fulfill this task. In particular, the SVOR approach with ``explicit constraints on thresholds'' (SVOREX) of \cite{Chu2007a} obtains a good trade-off between training speed and generalization capability. A dual formulation of SVOREX is presented in Equations~(\mbox{\ref{SVOR:1})--(\ref{SVOR:5})}. $\mathcal{K}$ is the kernel function, corresponding to a dot product in the kernel space, and $n^j$ is the number of samples in a class $j \in \{1,\dots,r\}$.
Each dual variable $\alpha_i^j$ takes a non-null value only when the $i^\text{th}$ sample of the $j^\text{th}$ class is active in the definition of the $j^\text{th}$ hyperplane, for $j \in \{1,\dots,r-1\}$. Similarly, each dual variable $\alpha_i^{*j}$ takes a non-null value only when the $i^\text{th}$ sample of the $j^\text{th}$ class is active in the definition of the $(j-1)^\text{th}$ hyperplane, for $j \in \{2,\dots,r\}$.
Additional constraints and variables~$\mu^j$ impose an order on the hyperplanes. For the sake of simplicity, the dummy variables $\boldsymbol\alpha^{*1}$, $\boldsymbol\alpha^{r}$, $\mu^1$, and $\mu^r$ are defined and should be fixed to zero.
\begin{equation}
\max_{\boldsymbol\alpha,\boldsymbol\alpha^*,\boldsymbol\mu} \sum_{j=1}^{r} \sum_{i=1}^{n^j} (\alpha_i^j + \alpha_i^{*j})  - \frac{1}{2} \sum_{j=1}^{r} \sum_{i=1}^{n^j} \sum_{j'=1}^{r} \sum_{i'=1}^{n^{j'}}  ( \alpha_i^{*j} - \alpha_{i}^{j} )  (\alpha_{i'}^{*j'} - \alpha_{i'}^{j'})  \mathcal{K}(x_i^j,x_{i'}^{j'}) \vspace*{-0.45cm} \label{SVOR:1}
\end{equation}
\begin{align}
 \text{s.t.}  \hspace*{0.3cm} & 0 \leq \alpha_{i}^{j}  \leq C&  \hspace*{1cm} j \in \{1,\dots,r\}, i \in \{1,\dots,n^j\}  \label{SVOR:2}  \\
& 0 \leq \alpha_{i}^{*j}  \leq C&  \hspace*{1cm}   j \in \{1,\dots,r-1\},  i \in \{1,\dots,n^j\}  \label{SVOR:3} \\
& \sum_{i=1}^{n^j} \alpha_i^j + \mu^j = \sum_{i=1}^{n^{j+1}}  \alpha_i^{*j+1} + \mu^{j+1} &   j \in \{1,\dots,r-1\} \label{SVOR:4} \\
& \mu^j \geq 0 &   j \in \{1,\dots,r-1\}. \label{SVOR:5}
\end{align}

The last two constraints of \mbox{Equations~(\ref{SVOR:4})--(\ref{SVOR:5})} can be reformulated to eliminate the $\boldsymbol\mu$ variables, leading to nested constraints on the variables $\boldsymbol\alpha$ and $- \boldsymbol\alpha^*$:
\begin{align}
&\sum_{k=1}^j \left( \sum_{i=1}^{n^k} \alpha_i^k - \sum_{i=1}^{n^{k+1}}  \alpha_i^{*k+1}  \right) \geq 0 \hspace*{1.5cm}   j \in \{1,\dots,r-2\}  \label{SVORnew:1} \\
&\sum_{k=1}^{r-1} \left( \sum_{i=1}^{n^k} \alpha_i^k - \sum_{i=1}^{n^{k+1}}  \alpha_i^{*k+1}  \right) = 0. \hspace*{1.5cm} \label{SVORnew:2}
\end{align}

Overall, the problem of Equations~(\mbox{\ref{SVOR:1})--(\ref{SVOR:3})} and (\mbox{\ref{SVORnew:1})--(\ref{SVORnew:2})} is a nonseparable convex problem over the same constraint polytope as the RAP--NC. Note that the number of nested constraints, corresponding to the number of classes, is usually much smaller than the number of variables, which is proportional to the total number of samples, and thus~$m \ll n$.

The solutions of this formulation are usually sparse, since only a fraction of the samples (support vectors) define the active constraints and separating hyperplanes. 
Given this structure and the size of practical applications, modern solution methods rely on decomposition steps, in which a working set of variables is iteratively re-optimized by a method of choice. Such an approach is referred to as \emph{block-coordinate descent} in \cite{Bertsekas2003}. The convergence of the algorithm can be guaranteed by including in the working set the variables that most severely violate the KKT conditions. \cite{Chu2007a}, in line with the work of \cite{Platt1998}, consider a minimal working set with only two variables at each iteration. The advantage is that the subproblem can be solved analytically in this case, the disadvantage is that a large number of working set selections can be needed for convergence, and the KKT condition check and gradient update may become the bottleneck instead of the optimization itself. To better balance the computational effort and reduce the number of decomposition steps, larger working sets could be considered (e.g., as in SVM$^\textsc{light}$ of \citealt{Joachims1999}). Still, to be successful, the algorithm must solve each subproblem, here a non-separable RAP--NC, very efficiently. Such an alternative optimization approach will be investigated in Section \ref{ordinal}.\\ \vspace*{-0.2cm}

\noindent
\textbf{Portfolio Management.}
The mean-variance optimization (MVO) model of \cite{Markowitz1952} has been refined over the years to integrate a large variety of constraints.
In its most classical form,
the model aims to maximize expected return while minimizing a risk measure such as the variance of the return. This problem can be formulated as:
\begin{align}
 & \left\{ \max   \sum_{i=1}^n x_i \mu_i \ ; \  \min \sum\limits_{i=1}^{n} \sum\limits_{j=1}^{n} x_i x_j \sigma_{ij}   \right\} & \label{Portfolio:1} \\
\text{s.t.}  \hspace*{0.3cm} & \sum_{i=1}^n x_i = 1  \label{Portfolio:2} \\
  &0 \leq \ x_i & i  \in \{ 1,\dots,n\}, \label{Portfolio:3}
\end{align}
where the $x_i$ variables, $i \in \{1,\dots,n\}$, represent investments in different assets, $\mu_i$ is the expected return of asset $i$, and $\sigma_{ij}$ the covariance between asset $i$ and $j$. In this model, Equation (\ref{Portfolio:2}) is used to normalize the total investment and Equation (\ref{Portfolio:3}) prevents short-selling. 
The literature on these models is vast, and we refer to the recent surveys of \cite{Kolm2014} and \cite{Mansini2014} for more thorough descriptions. Two additional constraint families, often used in practical portfolio models, are closely linked with the RAP--NC:
\begin{itemize}
\item \emph{Class constraints} limit the investment amounts for certain classes of assets or economic sectors. These can result from regulatory requirements, managerial insights, or customer guidelines (see, e.g., \citealt{Chang2000} and \citealt{Anagnostopoulos2010}). The assets may also be ranked into different categories, e.g., based on their risk or ecological impact. Imposing investment bounds at each level leads to the nested constraints of Equation~(\ref{LBUB2:2}).

\item \emph{Fixed transaction costs, minimum transaction levels, and cardinality constraints} either impose a fixed price or threshold quantity for any investment in an asset, or limit the number of positions on different assets. These constraints usually require the introduction of additional integer variables $y_i$, taking value one if and only if the asset $i$ is included in the portfolio. 
This leads to quadratic MIPs, for which metaheuristics \citep{Chang2000,Crama2003} and branch-and-cut methods \citep{Bienstock1996,Jobst2001} form the current state-of-the-art. 
\cite{Bienstock1996} branches on the $y_i$ variables and solves a quadratic resource allocation problem, with additional surrogate constraints in the form of Equation (\ref{LBUB2:2}), at each node of the search tree. Improved algorithms for the RAP--NC can thus also prove helpful as a methodological building block for more complex portfolio optimization algorithms.\\ \vspace*{-0.2cm}
\end{itemize}

\noindent
\textbf{Telecommunications.}
Constrained resource allocation problems also have a variety of applications in telecommunications. Mobile signals, for example, can be emitted in different directions with different power levels, but interference between signals emitted in the same direction reduces the quality of the communication. In this context, a power and direction must be determined for each signal, while respecting service-quality constraints and minimizing transmission costs. As underlined by \cite{Viswanath2002} and \cite{Padakandla2009a}, this problem can be formulated as an instance of the RAP--NC. Given the large size of typical applications, the efficiency of the algorithm is of foremost importance.

A similar model arises for power minimization in multiple-input and multiple-output communication systems, as well as in various other applications of optimization to telecommunications \citep{DAmico2014a}. Moreover, the RAP--NC generalizes a family of~multilevel \emph{water-filling} problems, which have been the focus of significant research \citep{Palomar2005}. Other applications include power management on multimedia devices, discussed by \cite{Huang2009} and \cite{Gerards2016}.
As illustrated by these example applications, the RAP--NC is a prototypical model and an elementary building block for various problems. Therefore, a new algorithmic breakthrough can have considerable impact in many contexts.\\ \vspace*{-0.2cm}
 
\section{Proposed Methodology}
\label{section:method}

In this section, we first describe the proposed methodology for the case of continuous variables, and then move on to the case with integer variables.
We assume that $a_i \leq b_i$ for $i \in \{1,\dots,n\}$, otherwise the problem is trivially infeasible. We will use boldface notation for vectors and normal font for scalars. Let $\mathbf{e}^s$ be the unit vector such that $e_s = 1$ and $e_i = 0$ for $i \neq s$.

\subsection{Continuous RAP--NC}

The proposed algorithm for the RAP--NC is a divide-and-conquer approach over the indices of the nested constraints. It can be seen as a generalization of the method of \cite{Vidal2014a}, with some fundamental differences related to the number and the nature of the subproblems.
For each range of indices $(v,w)$ considered during the search, such that $1 \leq v \leq w \leq m$, it solves four subproblems, $\text{RAP--NC}_{v,w}(L,R)$ for $L \in \{a_{v-1},b_{v-1}\}$ and $R \in \{a_w,b_w\}$, expressed in \mbox{Equations (\ref{rapnc11})--(\ref{rapnc13})}, obtained by fixing the $(v-1)^\text{th}$ and $w^\text{th}$ nested constraints to their lower or upper bounds. $M$ is a large number, defined to be larger than the Lipschitz constant of each function $f_i$.\vspace*{-0.2cm}
\begin{align}
\text{RAP--NC}_{v,w}(L,R): \hspace*{0.3cm}  \text{min} \  \  &\bar{f}(\mathbf{x}) = \sum_{i=\sigma[v-1]+1}^{\sigma[w]} \bar{f}_i(x_i)   \hspace*{-3.8cm} & \label{rapnc11} \\
\text{s.t.} \hspace*{0.6cm}  a_i - L \leq  &\sum_{k=\sigma[v-1]+1}^{\sigma[i]} x_k \leq \ b_i - L \hspace*{-3.8cm}  & i \in \{v, \dots, w-1\} \label{rapnc12b} \\
&\sum_{i=\sigma[v-1]+1}^{\sigma[w]} x_i = \ R - L   \hspace*{-3.8cm} & \label{rapnc12} \\
&\text{with } \bar{f}_i(x) = 
 \begin{cases}
f_i(c_i) + M (c_i - x) & \text{ if } x < c_i  \\
f_i(x_i) & \text{ if } x \in [c_i,d_i]  \\
f_i(d_i) + M (x - d_i) & \text{ if } x > d_i 
\end{cases}   \hspace*{-3.8cm} \label{rapnc13}
\end{align}

To solve these problems when $v < w$, the algorithm relies on known optimal solutions obtained deeper in a recursion over the range~$(v,u)$ and $(u+1,w)$, with \mbox{$u =  \lfloor (v+w)/2 \rfloor$}. When $v = w$ (at the bottom of the recursion), the $\text{RAP--NC}_{v,v}(L,R)$ does not contain any nested constraints from Equation (\ref{rapnc12b}) and thus reduces to a simple RAP.
We will refer to this approach as the \emph{monotonic decomposition algorithm}, \textsc{MDA}$(v,w)$.
The original RAP--NC is solved by \textsc{MDA}$(1,m)$, and the maximum depth of the recursion is $\lceil \log m \rceil$ since the binary decomposition is performed over the $m$ nested constraints.

In the formulation given by \mbox{Equations (\ref{rapnc11})--(\ref{rapnc13})}, note that the bounds $c_i \leq x_i \leq d_i$ are transferred into the objective via an exact L1 penalty function. This is possible since 
the functions~$f_i$ satisfy the Lipschitz condition (Theorem \ref{reformulation}), and it helps simplify the exposition and proofs.

\begin{theorem}[\textbf{Relaxation--Penalization}]
\emph{
If there exists a solution $x$ of $\text{RAP--NC}_{v,w}(L,R)$ such that $\mathbf{c} \leq \mathbf{x} \leq \mathbf{d}$, then all optimal solutions of $\smash{\text{RAP--NC}_{v,w}(L,R)}$ satisfy $\mathbf{c} \leq \mathbf{x} \leq \mathbf{d}$.
} \label{reformulation}
\end{theorem}

\proof
\begin{leftbar} 
Assume the existence of an optimal solution $\mathbf{x^*}$ of $\text{RAP--NC}_{v,w}(L,R)$ with an index~$s \in \{\sigma[v-1]+1,\dots,\sigma[w]\}$ such that $x^*_s > d_s$, and a solution $\mathbf{x}$ such that $\mathbf{c} \leq \mathbf{x} \leq \mathbf{d}$. Since \mbox{$x^*_s > d_s \geq x_s$} and  $\sum_{k=\sigma[v-1]+1}^{\sigma[w]} x^*_k = \sum_{k=\sigma[v-1]+1}^{\sigma[w]} x_k = R - L$, either\vspace*{-0.2cm}
\begin{equation}
\label{testeq}
\sum_{k=\sigma[v-1]+1}^{s}  \hspace*{-0.2cm} x^*_k \hspace*{0.2cm} > \hspace*{-0.2cm} \sum_{k=\sigma[v-1]+1}^{s}  \hspace*{-0.2cm}  x_k \ \hspace*{1cm} \text{ or } \  \hspace*{1cm} \sum_{k=s}^{\sigma[w]}  x^*_k > \sum_{k=s}^{\sigma[w]}  x_k.\vspace*{-0.2cm}
\end{equation}

In the first case, define $t = \min \{ i \ |  \ i > s \text{ and } \sum_{k=\sigma[v-1]+1}^i x^*_k \leq \sum_{k=\sigma[v-1]+1}^i x_k\}$. Observe that $x^*_t < x_t$ and thus $d_t - x^*_t > 0$.
Moreover, there exists $\Delta > 0$ such that, for each $j$ such that  $\sigma[j] \in \{s,\dots,t-1\},$ $ \smash{\sum_{k=\sigma[v-1]+1}^{\sigma[j]} x^*_k - \Delta} > \smash{\sum_{k=\sigma[v-1]+1}^{\sigma[j]} x_k \geq a_i}$.
Defining $\Delta' = \min\{\Delta,  d_t - x^*_t, x^*_s - d_s \}$,
the solution $\mathbf{x'^*} = \mathbf{x^*} + \Delta' ( \mathbf{e}^t -\mathbf{e}^s)$ is feasible and such that $\bar{f}(\mathbf{x'^*}) = \bar{f}(\mathbf{x^*}) + \bar{f}_t(x^*_t + \Delta') - \bar{f}_t(x^*_t) + \bar{f}_s(x^*_s - \Delta') - \bar{f}_s(x^*_s).$ Due to the Lipschitz condition, we have $\bar{f}_t(x^*_t + \Delta') - \bar{f}_t(x^*_t) < M \Delta'$. Moreover, $M \Delta' = \bar{f}_s(x^*_s) -  \bar{f}_s(x^*_s - \Delta')$ and thus  $\bar{f}(\mathbf{x'^*}) < \bar{f}(\mathbf{x^*})$, contradicting the optimality of $\mathbf{x^*}$.

The second case of Equation (\ref{testeq}) is analogous.\qedhere
\end{leftbar}
\endproof

The main challenge of the \textsc{MDA} is now to exploit the information gathered at deeper steps of the recursion to solve each $\text{RAP--NC}$ efficiently. For this purpose, we introduce Theorem~\ref{p1}, which expresses a monotonicity property of the optimal solutions as a function of the resource bound $R$. As shown subsequently in Theorem \ref{p1b}, this result allows to generate tighter bounds on the variables, which supersede the nested constraints of the RAP--NC and allow to solve all subproblems (at all recursion levels) as simple RAPs.

\begin{theorem}[\textbf{Monotonicity}]
\emph{
Consider three bounds $R^{\downarrow} \leq R \leq R^{\uparrow}$.
If $\mathbf{x^\downarrow}$ is an optimal solution of $\text{RAP--NC}_{v,w}(L,R^{\downarrow})$ and $\mathbf{x^\uparrow}$ is an optimal solution of $\text{RAP--NC}_{v,w}(L,R^{\uparrow})$ such that $\mathbf{x^\downarrow} \leq \mathbf{x^\uparrow}$, then there exists an optimal solution $\mathbf{x^*}$ of  $\text{RAP--NC}_{v,w}(L,R)$ such that $\mathbf{x^\downarrow} \leq \mathbf{x^*} \leq \mathbf{x^\uparrow}$.
}
\label{p1}
\end{theorem}

\proof
\begin{leftbar}
Define $\bar{a}_i = a_i - L$ and $\bar{b}_i = b_i - L$ for $i \in \{v,\dots,w-1\}$ as well as \mbox{$\bar{a}_w = \bar{b}_w = R - L$}.
By the KKT conditions (in the presence of a convex objective over a set of linear constraints), if $\mathbf{x}$ is an optimal solution of $\text{RAP--NC}_{v,w}(L,R)$, then there exist dual multipliers $(\boldsymbol\kappa, \boldsymbol\lambda)$ such that:
\begin{align}
& \Phi_i =   \hspace*{-0.3cm} \sum_{k  \in \{v, \dots, w\} | \sigma[k] \geq i} \hspace*{-0.3cm} (\kappa_k - \lambda_k)  \in  \partial \bar{f}_i(x_i)  & \hspace{1cm} i \in \{\sigma[v-1]+1, \dots, \sigma[w]\}  \label{KKT3} \\
& \bar{a}_i \leq  \sum\limits_{k=\sigma[v-1]+1}^{\sigma[i]}  x_k   \leq  \  \bar{b}_i   &i  \in \{ v,\dots,w\} \label{KKT1} \\
& \kappa_i \left( \sum_{k=\sigma[v-1]+1}^{\sigma[i]} x_k - \bar{a}_i \right) = 0,   \kappa_i \in \Re^+ &  \hspace{1cm} i \in \{v, \dots, w\}  \label{KKT4} \\
& \lambda_i \left( \bar{b}_i - \sum_{k=\sigma[v-1]+1}^{\sigma[i]} x_k\right) = 0,   \lambda_i \in \Re^+ &  \hspace{1cm} i \in \{v, \dots, w\} \label{KKT5}
\end{align}

Note the appearance of the subgradients $\partial \bar{f}_i$ in Equation (\ref{KKT3}), since the functions~$\bar{f}_i$ are not necessarily differentiable. Let $(\boldsymbol\kappa^\uparrow, \boldsymbol\lambda^\uparrow, \boldsymbol\Phi^\uparrow)$ be a set of multipliers associated with the optimal solution $\mathbf{x^{\uparrow}}$ of $\text{RAP--NC}_{v,w}(L,R^{\uparrow})$, and  $\mathbf{x}$ be an optimal solution of $\text{RAP--NC}_{v,w}(L,R)$. Define $\mathcal{S}^+_{\mathbf{x}} = \{i \ | \ x_i > x^{\uparrow}_i \}$, $\mathcal{S}^-_{\mathbf{x}}= \{i \ | \ x_i < x^{\downarrow}_i\}$, and  $\mathcal{S}_{\mathbf{x}}= \{i \ | \  x^{\downarrow}_i \leq x_i \leq x^{\uparrow}_i \}$.
We will present a construct that generates a sequence of solutions $(\mathbf{x}^{\textsc{k}})$, starting from $\mathbf{x}^{0} = \mathbf{x}$, such that 
$\big|\mathcal{S}^+_{\mathbf{x}^{\textsc{k}+1}}\big| < \big|\mathcal{S}^+_{\mathbf{x}^{\textsc{k}}}\big|$
and 
$\big|\mathcal{S}^-_{\mathbf{x}^{\textsc{k}+1}}\big| \leq \big|\mathcal{S}^-_{\mathbf{x}^{\textsc{k}}}\big|$ 
as long as $\big|\mathcal{S}^+_{\mathbf{x}^{\textsc{k}}}\big| > 0$, leading by recurrence to a solution $\mathbf{\bar{x}}$ such that~$\mathbf{\bar{x}} \leq \mathbf{x}^{\uparrow}$.

If $\big|\mathcal{S}^+_{\mathbf{x}^{\textsc{k}}}\big| > 0$, then there exists $s \in \{\sigma[v-1]+1, \dots, \sigma[w]\}$ such that~$x^{\uparrow}_s < x^{\textsc{k}}_s$.
Let~$r$ be the greatest index in $\{\sigma[v-1]+1,\dots,s\}$ such that $\sum_{k=\sigma[v-1]+1}^{r-1} x^{\textsc{k}}_k \geq \sum_{k=\sigma[v-1]+1}^{r-1} x^{\uparrow}_k$, and let~$t$ be the smallest index in $\{s,\dots,\sigma[w]\}$ such that $\smash{\sum_{k=\sigma[v-1]+1}^{t} x^{\textsc{k}}_k \leq \sum_{k=\sigma[v-1]+1}^{t} x^{\uparrow}_k}$.

Since $\smash{R^\uparrow - L = \sum_{i=\sigma[v-1]+1}^{\sigma[w]} x^{\uparrow}_i \geq \sum_{i=\sigma[v-1]+1}^{\sigma[w]} x^{\textsc{k}}_i = R - L}$, and by the definition of $r$~and~$t$, it follows that $\smash{\sum_{i=r}^t x^{\uparrow}_i \geq \sum_{i=r}^t x^{\textsc{k}}_i}$. Moreover, $r < s \Rightarrow x^{\textsc{k}}_r < x^{\uparrow}_r$, and $s < t \Rightarrow  x^{\textsc{k}}_t  < x^{\uparrow}_t$. Finally, note that $r = s =t$ (jointly) is impossible.

\begin{itemize}[nosep]

\item
When $r < s$, the following statements are valid:

\noindent
For each $j$ such that  $\sigma[j]   \in \{ r,\dots, s-1\}$, 
 $\smash{\bar{a}_j \leq \sum_{k=\sigma[v-1]+1}^{\sigma[j]} x^{\textsc{k}}_k < \sum_{k=\sigma[v-1]+1}^{\sigma[j]} x^{\uparrow}_k \leq \bar{b}_j}$ (by the definition of $r$) and thus \mbox{$\kappa^{\uparrow}_j = \lambda^{\textsc{k}}_j = 0$}. As a consequence, $\Phi^{\textsc{k}}_{i} \geq \Phi^{\textsc{k}}_{i+1}$ and $\Phi^{\uparrow}_{i} \leq \Phi^{\uparrow}_{i+1}$  for \mbox{$i \in \{r,\dots,s-1\}$}.
The functions $\bar{f}_i$ are convex, and thus their (Clarke) subgradients are monotone \mbox{\citep{Rockafellar1970}}, i.e., $\{ x^{\uparrow}_s < x^{\textsc{k}}_s, \Phi^{\uparrow}_{s} \in \partial \bar{f}_s(x^{\uparrow}_s), \Phi^{\textsc{k}}_{s} \in \partial \bar{f}_s(x^{\textsc{k}}_s) \} \Rightarrow \Phi^{\uparrow}_{s}   \leq  \Phi^{\textsc{k}}_{s}  $.  Similarly, we have $\{ x^{\uparrow}_r > x^{\textsc{k}}_r, \Phi^{\uparrow}_{r} \in \partial \bar{f}_r(x^{\uparrow}_r), \Phi^{\textsc{k}}_{r} \in \partial \bar{f}_r(x^{\textsc{k}}_r) \} \Rightarrow \Phi^{\uparrow}_{r}  \geq  \Phi^{\textsc{k}}_{r}$.
Combining these relations leads to 
\begin{equation} 
\Phi^{\textsc{k}}_s \leq \Phi^{\textsc{k}}_r \leq \Phi^{\uparrow}_r \leq  \Phi^{\uparrow}_s \leq \Phi^{\textsc{k}}_s,  \label{argument1}
\end{equation}
and thus there exists $\Psi \in \Re$ such that $\Phi^\uparrow_i = \Phi^{\textsc{k}}_i = \Psi$  for \mbox{$i \in \{r,\dots,s\}$}.

\item
\hspace*{-0.5cm} $\bullet$
When $s < t$, the following statements are valid:

\noindent
For each  $j$ such that  $\sigma[j]   \in \{ s,\dots, t-1\}$, 
$\smash{\bar{a}_j \leq \sum_{k=\sigma[v-1]+1}^{\sigma[j]} x^{\uparrow}_k < \sum_{k=\sigma[v-1]+1}^{\sigma[j]} x^{\textsc{k}}_k \leq \bar{b}_j}$  (by the definition of $t$)  and thus $\lambda^{\uparrow}_j = \kappa^{\textsc{k}}_j = 0$. As a consequence, $\Phi^{\textsc{k}}_{i} \leq \Phi^{\textsc{k}}_{i+1}$ and $\Phi^{\uparrow}_{i} \geq \Phi^{\uparrow}_{i+1}$  for $i \in \{s,\dots,t-1\}$.
Furthermore, as before, $x^{\uparrow}_s < x^{\textsc{k}}_s$ and  $x^{\uparrow}_t > x^{\textsc{k}}_t$, and thus $\Phi^{\uparrow}_{s}  \leq \Phi^{\textsc{k}}_{s} $ and $\Phi^{\uparrow}_{t}  \geq \Phi^{\textsc{k}}_{t} $.
Combining these relations leads to 
\begin{equation}
\Phi^{\textsc{k}}_s \leq \Phi^{\textsc{k}}_t \leq \Phi^{\uparrow}_t \leq  \Phi^{\uparrow}_s \leq \Phi^{\textsc{k}}_s, \label{argument2}
\end{equation}
and thus there exists $\Psi \in \Re$ such that $\Phi^\uparrow_i = \Phi^{\textsc{k}}_i = \Psi$  for \mbox{$i \in \{s,\dots,t\}$}.
 \end{itemize}

Overall, $\Phi^\uparrow_i = \Phi^{\textsc{k}}_i = \Psi$ for \mbox{$i \in \{r,\dots,t\}$}, and thus $\Psi \in \partial  \bar{f}_i(x^{\textsc{k}}_i) \cap \partial \bar{f}_i(x^{\uparrow}_i)$ for $i \in \{r,\dots,t\}$. Define $x_i^\textsc{min} = \min\{x^{\textsc{k}}_i,x^{\uparrow}_i\}$ and $x_i^\textsc{max} = \max\{x^{\textsc{k}}_i,x^{\uparrow}_i\}$. This implies that $\partial  \bar{f}_i(x) = \{ \Psi \}$ for $x \in (x_i^\textsc{min},x_i^\textsc{max})$ and thus these functions are affine with identical slope: $\bar{f}_i(x) =  \bar{f}_i(x^{\textsc{k}}_i) + \Psi (x - x^{\textsc{k}}_i)$
for $x \in [x_i^\textsc{min},x_i^\textsc{max}]$. We can thus transfer value from the variables of the set $\mathcal{S}^+ = \mathcal{S}^+_{\mathbf{x}^{\textsc{k}}} \cap \{r,\dots,t\}$ to those of the set $ \overline{\mathcal{S}^+}  = \{r,\dots,t\} - \mathcal{S}^+$, via $\textsc{Adjust}([r,\dots,t],\mathbf{x}^{\textsc{k}},\mathbf{x^\uparrow})$
(Algorithm \ref{algo:transfer}),
leading to a feasible solution $\mathbf{x}^{\textsc{k}+1}$ with the same cost as $\mathbf{x^{\textsc{k}}}$, hence optimal, such that
$$
\begin{cases}
x^{\textsc{k}+1}_i = x^{\uparrow}_i & \text{for } i \in \mathcal{S}^+ \\
x^\textsc{k}_i \leq x^{\textsc{k}+1}_i \leq x^{\uparrow}_i & \text{for } i \in \overline{\mathcal{S}^+} \\
x^{\textsc{k}+1}_i = x^\textsc{k}_i & \text{otherwise}. \\
\end{cases}
$$
We observe that $\big|\mathcal{S}^+_{\mathbf{x}^{\textsc{k}+1}}\big| < \big|\mathcal{S}^+_{\mathbf{x}^{\textsc{k}}}\big|$, moreover
 $\big|\mathcal{S}^-_{\mathbf{x}^{\textsc{k}+1}}\big| \leq \big|\mathcal{S}^-_{\mathbf{x}^{\textsc{k}}}\big|$.
By recurrence, repeating the previous transformation leads to a solution $\mathbf{\bar{x}}$ such that 
 $\mathcal{S}^+_{\mathbf{\bar{x}}} = \varnothing$.
A similar principle can then be applied to 
generate a sequence of solutions $(\mathbf{\bar{x}}^\textsc{k})$, starting from $\mathbf{\bar{x}}^0 = \mathbf{\bar{x}}$, 
such that~\mbox{$\big|\mathcal{S}^-_{\mathbf{\bar{x}}^{\textsc{k}+1}}\big| < \big|\mathcal{S}^-_{\mathbf{\bar{x}}^{\textsc{k}}}\big|$}
and 
$\mathcal{S}^+_{\mathbf{\bar{x}}^{\textsc{k}+1}}  = \varnothing$
as long as $\big|\mathcal{S}^-_{\mathbf{\bar{x}}^{\textsc{k}}}\big| > 0$, leading to an optimal solution $\mathbf{x^{*}}$ such that~$\mathbf{x}^{\downarrow} \leq \mathbf{x}^{*} \leq \mathbf{x}^{\uparrow}$.\qedhere
\end{leftbar}
\endproof

\begin{algorithm}[htbp]
\linespread{1.3}\selectfont

$\Delta \gets 0$ \;
\For{$i=V_1,\dots,V_{|V|}$}
	{
	\If {$x_i > x_i^\uparrow$}
        {
               $\Delta \gets \Delta + x_i - x^\uparrow_i$ \;
              $x_i \gets x^\uparrow_i$ \;
             
	}
}
\For{$i=V_1,\dots,V_{|V|}$}
	{
		\If{$x_i < x_i^\uparrow$}
      	 {
                $\delta = \min \{x^\uparrow_i - x_i,\Delta\}$ \; 
               $x_i = x_i + \delta$ \;
		$\Delta = \Delta - \delta$ \;
	}
}

 \caption{$\textsc{Adjust}(V,\mathbf{x},\mathbf{x^\uparrow})$} \label{algo1} 
 \label{algo:transfer} 
\end{algorithm}

\begin{theorem}[\textbf{Variable Bounds}]
\label{p1b}
\emph{
Let $\mathbf{x^{L a}}$, $\mathbf{x^{L b}}$,  $\mathbf{x^{a R}}$, and  $\mathbf{x^{b R}}$ be optimal solutions of $\text{RAP--NC}_{v,u}(L,a_u)$, $\text{RAP--NC}_{v,u}(L,b_u)$,  $\text{RAP--NC}_{u+1,w}(a_u,R)$, and $\text{RAP--NC}_{u+1,w}(b_u,R)$, respectively. 
If $\mathbf{x^{L a}} \leq \mathbf{x^{L b}}$ and $\mathbf{x^{b R}} \leq \mathbf{x^{a R}}$, then there exists an optimal solution  $\mathbf{x^{*}}$ of $\text{RAP--NC}_{v,w}(L,R)$ such that:
\begin{align}
&x_i^{L a} \leq x^{*}_i \leq x_i^{L b}     &&\hspace*{-2cm} \text{ for } i \in \{\sigma[v-1]+1,\dots,\sigma[u]\},\text{ and} \label{mybounds} \\
&x_i^{b R} \leq x^{*}_i \leq x_i^{a R}    &&\hspace*{-2cm} \text{ for } i \in \{\sigma[u]+1,\dots,\sigma[w]\}. \label{mybound2}
\end{align}
}
\end{theorem}

\proof
\begin{leftbar}
Let $\mathbf{x}$ be an optimal solution of $\text{RAP--NC}_{v,w}(L,R)$. As such, $(x_{\sigma[v-1]+1},\dots, x_{\sigma[u]})$ and $(x_{\sigma[u]+1},\dots, x_{\sigma[w]})$ must be optimal solutions of $\text{RAP--NC}_{v,u}(L,X)$ and $\text{RAP--NC}_{u+1,w}(X,R)$ with $\smash{X = L + \sum_{i=\sigma[v-1]+1}^{\sigma[u]} x_i}$.
Since $a_u \leq X \leq b_u$, there exists an optimal solution $\mathbf{x^{*}}$ of $\text{RAP--NC}_{v,u}(L,X)$ such that $x_i^{L a} \leq x^{*}_i \leq x_i^{L b}$ for $i \in \{\sigma[v-1]+1, \dots,\sigma[u]\}$ via Theorem~\ref{p1}. The other inequality is obtained for $i \in \{\sigma[u]+1,\dots,\sigma[w]\}$ with the same argument, after re-indexing the variables downwards from $\sigma[w]$ to $\sigma[u]+1$.\qedhere
\end{leftbar}
\endproof

As a consequence of Theorems \ref{p1} and \ref{p1b}, the inequalities of Equations (\ref{mybounds})--(\ref{mybound2}) are valid and can be added to the RAP--NC formulation given by Equations (\ref{rapnc11})--(\ref{rapnc13}). Moreover, we show that these inequalities dominate the nested constraints of Equation (\ref{rapnc12b}). Indeed, 
\begin{equation}
\begin{aligned}
x^{L a}_k \leq x_k \leq x^{L b}_k \text{ for }   & k  \in \{ \sigma[v-1]+1,\dots,\sigma[u]\} \text{ and }  i  \in \{ v,\dots,u\} \\ &\Rightarrow      \sum_{k=\sigma[v-1]+1}^{\sigma[i]} x^{L a}_k \leq  \sum_{k=\sigma[v-1]+1}^{\sigma[i]} x_k \hspace*{0.2cm}   \leq   \sum_{k=\sigma[v-1]+1}^{\sigma[i]} x^{L b}_k \\ 
&\Rightarrow \hspace*{1.3cm} \bar{a}_i  \hspace*{0.5cm} \leq   \sum_{k=\sigma[v-1]+1}^{\sigma[i]} x_k  \hspace*{0.2cm}  \leq  \hspace*{0.4cm} \bar{b}_i  \hspace*{0.4cm}  \text{ and}
\end{aligned}
\end{equation}

\begin{equation}
\begin{aligned}
x^{b R}_k \leq x_k \leq x^{a R}_k \text{ for }   &k  \in \{ \sigma[u]+1,\dots,\sigma[w]\} \text{ and }  i  \in \{ u,\dots,w-1\}  \\ &\Rightarrow     \hspace*{0.15cm}\sum_{k=\sigma[i]+1}^{\sigma[w]}\hspace*{0.15cm}  x^{b R}_k  \hspace*{0.1cm} \leq  \hspace*{0.15cm}\sum_{k=\sigma[i]+1}^{\sigma[w]}\hspace*{0.15cm}  x_k \hspace*{0.15cm}   \leq  \hspace*{0.15cm}\sum_{k=\sigma[i]+1}^{\sigma[w]}\hspace*{0.15cm}  x^{a R}_k. \end{aligned}
\label{totot}
\end{equation}
Moreover, Equations (\ref{mybounds})--(\ref{mybound2}) imply that:
\begin{equation}
\sum_{k=\sigma[v-1]+1}^{\sigma[u]} \hspace*{-0.2cm} x^{L b}_k + \sum_{k=\sigma[u]+1}^{\sigma[w]}  \hspace*{-0.2cm} x^{b R}_k =  \sum_{k=\sigma[v-1]+1}^{\sigma[w]}  \hspace*{-0.2cm} x_k = \sum_{k=\sigma[v-1]+1}^{\sigma[u]}  \hspace*{-0.2cm} x^{L a}_k + \sum_{k=\sigma[u]+1}^{\sigma[w]}  \hspace*{-0.2cm} x^{a R}_k = R - L, \label{totot2}
\end{equation}
and combining Equation (\ref{totot}) and (\ref{totot2}) leads to:
\begin{equation}
\begin{aligned}
&\Rightarrow  \sum_{k=\sigma[v-1]+1}^{\sigma[u]} x^{L b}_k + \sum_{k=\sigma[u]+1}^{\sigma[i]}  x^{b R}_k \hspace*{0.2cm} \geq  \sum_{k=\sigma[v-1]+1}^{\sigma[i]}  x_k \hspace*{0.2cm} \geq  \sum_{k=\sigma[v-1]+1}^{\sigma[u]}  \hspace*{-0.2cm} x^{L a}_k + \sum_{k=\sigma[u]+1}^{\sigma[i]}  \hspace*{-0.2cm} x^{a R}_k \\
&\Rightarrow \hspace*{4.4cm} \bar{b}_i  \hspace*{0.2cm} \geq   \sum_{k=\sigma[v-1]+1}^{\sigma[i]} x_k  \hspace*{0.2cm}  \geq  \hspace*{0.3cm} \bar{a}_i.
\end{aligned}
\end{equation}

Therefore, the nested constraints are superseded at each level of the recursion by the variable bounds obtained from the subproblems. The immediate consequence is a problem simplification: without nested constraints, the formulation reduces to a simple RAP given in \mbox{Equations~(\ref{simplerap11})--(\ref{simplerap13})},~which can be efficiently solved by the algorithm of \cite{Frederickson1982} or \cite{Hochbaum1994}.
\begin{align}
\hspace*{-0.5cm}\text{RAP}_{v,w}(L,R,\mathbf{\bar{c}},\mathbf{\bar{d}}): \hspace*{0.3cm}
\text{min} \  \  \bar{f}(\mathbf{x}) = & \sum_{i=\sigma[v-1]+1}^{\sigma[w]} \bar{f}_i(x_i)   \hspace*{-0.8cm} & \label{simplerap11} \\
\text{s.t.} \hspace*{0.3cm} & \sum_{i=\sigma[v-1]+1}^{\sigma[w]} x_i = \ R - L   \hspace*{-0.8cm} & \label{simplerap12} \\
&\bar{c}_i \leq x_i \leq  \bar{d}_i & i  \in \{ \sigma[v-1]+1,\dots,\sigma[w]\}. \label{simplerap13}
\end{align}

The pseudocode of the overall decomposition approach is summarized in Algorithm~\ref{algo-subprobSC}.

\begin{algorithm}[htb]
\linespread{1.05}\selectfont
 \eIf{$v = w$}
{

 $(x_{\sigma[v-1]+1}^{aa},\dots,x_{\sigma[v]}^{aa}) \gets \textsc{Rap}_{v,v}(a_{v-1},a_{w},-\infty,\infty)$ \;
 $(x_{\sigma[v-1]+1}^{ab},\dots,x_{\sigma[v]}^{ab}) \gets \textsc{Rap}_{v,v}(a_{v-1},b_{w},-\infty,\infty)$ \;
 $(x_{\sigma[v-1]+1}^{ba},\dots,x_{\sigma[v]}^{ba}) \gets \textsc{Rap}_{v,v}(b_{v-1},a_{w},-\infty,\infty)$ \;
 $(x_{\sigma[v-1]+1}^{bb},\dots,x_{\sigma[v]}^{bb}) \gets \textsc{Rap}_{v,v}(b_{v-1},b_{w},-\infty,\infty)$ \;

}
{
$u  \gets  \lfloor \frac{v+w}{2} \rfloor$ \;

$\textsc{\textsc{MDA}}(v,u) $ \;

$\textsc{\textsc{MDA}}(u+1,w) $ \;

 \For{$(L,R) \in \{(a,a),(a,b),(b,a),(b,b)\}$}
{
\lIf{$\mathbf{x^{L a}} \nleq  \mathbf{x^{L b}}$}
{$\mathbf{x^{L a}} \gets \textsc{Adjust}([\sigma[v-1]+1,\dots,\sigma[u]],\mathbf{x^{L a}},\mathbf{x^{L b}})$}
 \For{$i = \sigma[v-1]+1$ to $\sigma[u]$}
{ 
$ [\bar{c}_i, \bar{d}_i]   \gets [x^{L a}_i, x^{L b}_i]$ \; 
}
\lIf{$\mathbf{x^{b R}} \nleq  \mathbf{x^{a R}}$}
{$\mathbf{x^{b R}} \gets \textsc{Adjust}([\sigma[w],\dots,\sigma[u]+1],\mathbf{x^{b R}},\mathbf{x^{a R}})$}
 \For{$i = \sigma[u]+1$ to $\sigma[w]$}
{ 
$[\bar{c}_i,\bar{d}_i] \gets [x^{b R}_i, x^{a R}_i]$ \; 
}
 $(x^{L R}_{\sigma[v-1]+1},\dots,x^{L R}_{\sigma[w]})  \gets \textsc{Rap}_{v,w}(L,R,\mathbf{\bar{c}},\mathbf{\bar{d}})$ \;
}
}
 \caption{$\textsc{MDA}(v,w)$} \label{algo-subprobSC} 
\end{algorithm}

\noindent
Two final discussions follow.
\begin{itemize}
\item
First, observe the occurrence of Algorithm \ref{algo:transfer} (\textsc{Adjust} function, introduced in the proof of Theorem~\ref{p1}) before setting the RAP bounds. This $\cO(n)$ time function can only occur when the functions $f_i$ are not strictly convex; in these cases, the solutions of the subproblems may not directly satisfy $\mathbf{x^{L a}} \leq\mathbf{x^{L b}}$ and $\mathbf{x^{b R}} \leq  \mathbf{x^{a R}}$ because of possible ties between resource-allocation choices. Alternatively, one could also use a \emph{stable} RAP solver that guarantees that the solution variables increase monotonically with the resource bound.

\item
Second, note the occurrence of the L1 penalty function associated with the original variables' bounds $c_i$ and $d_i$ in $\bar{f}_i(x_i)$ while $\bar{c}_i$ and $\bar{d}_i$ are maintained as hard constraints. Indeed, some subproblems (e.g.,  $\text{RAP--NC}_{v,v+1}(b_v,a_{v+1})$ when $b_{v} \geq a_{v+1}$ and $\mathbf{c} = 0$) may not have a solution respecting the bounds $c_i$ and $d_i$. On the other hand, the $\bar{c}_i$ and $\bar{d}_i$  constraints can always be fulfilled, otherwise the original problem would be infeasible, and their validity is essential to guarantee the correctness of the algorithm.

\noindent
Nevertheless, since efficient RAP algorithms exist for some specific forms of the objective function, e.g., quadratic \citep{Brucker1984,Ibaraki1988}, we wish to avoid explicit penalty terms in the objective. Therefore we note that an optimal solution $\mathbf{x^*}$ of $\text{RAP}_{v,w}(L,R,\mathbf{\bar{c}},\mathbf{\bar{d}})$ can be obtained as follows:\vspace*{-0.2cm}
\begin{equation}
\mathbf{x^*} =
\begin{cases}
 \mathbf{c'} + \frac{(R - L) - \sum_{i=\sigma[v-1]+1}^{\sigma[w]} c'_{i}}{ \sum_{i=\sigma[v-1]+1}^{\sigma[w]} (\bar{c}_i - c'_i)} (\mathbf{\bar{c}}-\mathbf{c'})& \hspace*{0.5cm} \text{if  } \sum_{i=\sigma[v-1]+1}^{\sigma[w]} c'_{i} >  R-L   \\
  \mathbf{d'} + \frac{(R - L) - \sum_{i=\sigma[v-1]+1}^{\sigma[w]} d'_{i}}{ \sum_{i=\sigma[v-1]+1}^{\sigma[w]} (\bar{d}_i - d'_i)} (\mathbf{\bar{d}}-\mathbf{d'})&  \hspace*{0.5cm}  \text{if  } \sum_{i=\sigma[v-1]+1}^{\sigma[w]} d'_{i} <  R-L \\
\mathbf{x} &  \hspace*{0.5cm}  \text{otherwise}\vspace*{-0.2cm}
\end{cases} \label{myreformulate}
\end{equation}
where $c'_i = \max \{ \bar{c}_i, \min \{ c_i, \bar{d}_i \} \}$, $d'_i = \min \{ \bar{d}_i, \max \{ d_i, \bar{c}_i \} \}$, and  $\mathbf{x}$ is the solution of the same RAP with the hard constraints of Equation (\ref{hard-const}):\vspace*{-0.2cm}
\begin{equation}
c'_i  \leq x_i \leq d'_i \hspace*{2cm} i  \in \{ \sigma[v-1]+1,\dots,\sigma[w]\}. \label{hard-const}
\end{equation}
Thus, the penalty functions for $c_i$ and $d_i$ are taken into account by a $\cO(n)$ test during each RAP resolution, and they never appear in the objective. Experimentally, we observe that the subproblems that fall in the first two cases of Equation~(\ref{myreformulate}) are solved notably faster, since they do not even require finding the minimum of a convex function.
\end{itemize}

\subsection{Integer Optimization and Proximity}

The previous section has considered continuous decision variables and proven the validity of the algorithm when all the subproblems are \emph{solved to optimality}. Still, this proof is of limited practical utility for bit-complexity computational models, since the solutions of separable convex problems can involve irrational numbers (e.g., $\min f(x) = x^3 - 6x, x \geq 0$) which have no finite binary representation. Therefore, assuming that a subproblem is solved to optimality without any assumption on the shape of the functions is impracticable.

For this reason, most articles that present computational complexity results for convex resource allocation and network flow problems rely on the notion of $\epsilon$-approximate solutions, located in the proximity of a truly optimal but not necessarily representable solution. In a decomposition algorithm such as MDA, proving that the method produces an $\epsilon$-approximate solution for a given~$\epsilon$ would require to control the precision of the algorithm at each level of the recursion, which could be cumbersome. Therefore, we adopt another approach, typically used in scaling algorithms \citep{Hochbaum1994,Moriguchi2011}, which consists in proving the validity of the algorithm for integer variables, and using a proximity theorem between the integer and continuous solutions. By solving an integer problem scaled by an appropriate factor, and translating back the integer solution into a continuous solution, any desired $\epsilon$ precision can be achieved.

We define the functions $\bar{f}^\textsc{pl}_i(x) = \bar{f}_i( \lfloor x \rfloor ) +  ( x - \lfloor x \rfloor ) \times ( \bar{f}_i( \lceil x \rceil ) -  \bar{f}_i( \lfloor x \rfloor )  )$, which correspond to an inner linearization of the objective using as base the set of integer values. We call the linearized problem $\text{RAP--NC}^\textsc{pl}_{v,w}(L,R)$; it aims to find the minimum of $\bar{f}^\textsc{pl}(\mathbf{x}) = \sum_{i=\sigma[v-1]+1}^{\sigma[w]} \bar{f}^\textsc{pl}_i(x_i)$ subject to Equations (\ref{rapnc12b})--(\ref{rapnc13}).
Since $\bar{f}_i$ and $\bar{f}^\textsc{pl}_i$ coincide on the integer domain, the integer $\text{RAP--NC}_{v,w}(L,R)$ and $\text{RAP--NC}^\textsc{pl}_{v,w}(L,R)$ have the same set of optimal solutions. Beyond this, there is a close relationship between the solutions of the integer $\text{RAP--NC}^\textsc{pl}_{v,w}(L,R)$ and those of its continuous counterpart, as formulated in Theorem \ref{theo:reformulation}, allowing us to prove the validity of Algorithm~\ref{algo-subprobSC} for integer variables (Theorem~\ref{theo:integer}).

\begin{theorem}[\textbf{Reformulation}]
\label{theo:reformulation}
\emph{
Any optimal solution $\mathbf{x^*}$ of the integer $\text{RAP--NC}^\textsc{pl}_{v,w}(L,R)$ is also an optimal solution of the continuous $\text{RAP--NC}^\textsc{pl}_{v,w}(L,R)$.
}
\end{theorem}

\proof
\begin{leftbar}
By contradiction.
Suppose that $\mathbf{x^*}$ is not an optimal solution of the continuous $\text{RAP--NC}^\textsc{pl}_{v,w}(L,R)$. Hence, there exists $\mathbf{x}$ such that $\bar{f}^\textsc{pl}(\mathbf{x})  < \bar{f}^\textsc{pl}(\mathbf{x^*}) $, and the set $\{i \ | \  x_i - \lfloor x_i \rfloor > 0 \}$ contains at least two elements since $\smash{\sum_{i=\sigma[v-1]+1}^{\sigma[w]} x_i} = R-L \in  \mathbb{Z}$.
Let $s$ and $t$ be, respectively, the first and second indices in this set.
We know that the functions $f^\textsc{pl}_s$ and $f^\textsc{pl}_t$ are linear in $[ \lfloor x_s \rfloor, \lceil x_s \rceil ]$ and $[ \lfloor x_t \rfloor, \lceil x_t \rceil ]$, respectively, with slope $\Phi_s$ and $\Phi_t$.
Observe that the solution
\begin{equation}
\mathbf{x'} =
\begin{cases}
\mathbf{x} + \min \{\lceil x_s \rceil - x_s, x_t - \lfloor x_t \rfloor\} ( \mathbf{e}^s - \mathbf{e}^t) & \text{ if }  \Phi_s \leq \Phi_t, \\
\mathbf{x} + \min \{\lceil x_t \rceil - x_t,x_s - \lfloor x_s \rfloor\} ( \mathbf{e}^t - \mathbf{e}^s) & \text{ otherwise,}  \\
\end{cases}
\end{equation}
is feasible and such that $\bar{f}^\textsc{pl}(\mathbf{x'})  \leq \bar{f}^\textsc{pl}(\mathbf{x})$. Also, note that the number of non-integer values of $\mathbf{x'}$ has been strictly decreased (by one or two). Repeating this process, we obtain an integer solution~$\mathbf{x^{**}}$ such that $\bar{f}^\textsc{pl}(\mathbf{x^{**}}) \leq \bar{f}^\textsc{pl}(\mathbf{x}) < \bar{f}^\textsc{pl}(\mathbf{x^{*}})$. This contradicts the original assumption that $\mathbf{x^{*}}$ is an optimal solution of the integer $\text{RAP--NC}^\textsc{pl}_{v,w}(L,R)$.\qedhere
\end{leftbar}
\endproof

\begin{theorem}[\textbf{Integer variables}]
\label{theo:integer}
\emph{
Theorems \ref{reformulation}, \ref{p1}, \ref{p1b} and Algorithm \ref{algo-subprobSC} remain valid for RAP--NCs with integer variables.
}
\end{theorem}

\proof
\begin{leftbar}
The mathematical arguments used in these proofs are independent of the continuous or integer nature of the variables. Moreover, the solution transformation of Algorithm~\ref{algo:transfer} preserves the integrality of the variables. The only element that requires continuous variables is the use of the (necessary) KKT conditions in Equations (\ref{KKT3})--(\ref{KKT5}). However, as we have demonstrated in Theorem~\ref{theo:reformulation}, an optimal solution of the $\text{RAP--NC}^\textsc{pl}_{v,w}(L,R)$ with integer variables is also an optimal solution of the continuous $\text{RAP--NC}^\textsc{pl}_{v,w}(L,R)$. Thus, the KKT conditions with functions $f_i^\textsc{pl}$ are necessary, hence completing the proof.\qedhere
\end{leftbar}
\endproof

Finally, we exploit a proximity result for the solutions of the continuous and integer $\text{RAP--NC}^\textsc{pl}_{v,w}(L,R)$:

\begin{theorem}[\textbf{Proximity}]
\label{proximity theorem}
\emph{
For any integer optimal solution $\mathbf{x^*}$ of $\text{RAP--NC}$ with $n \geq 2$ variables, there is a continuous optimal solution $\mathbf{x}$ such that
\begin{equation}
\label{eqprox}
| x_i - x^*_i | < n-1, \text { for }  i \in \{1,\dots,n\}.
\end{equation}
}
\end{theorem}

This theorem allows us to search for an $\epsilon$-approximate solution of the continuous problem by defining an integer RAP--NC in which all parameters ($a_i$, $b_i$, $c_i$, $d_i$) have been scaled by a factor $\lceil n/\epsilon \rceil$, solving this problem, and transforming back the solution.
It constitutes a special case of Theorem 1.3 from \cite{Moriguchi2011}, as the $\text{RAP--NC}$ can be shown to be a special case of resource allocation problem under submodular constraints.
Moreover, without even relying on submodular optimization arguments, this result can also be obtained directly via first-order (KKT) optimality conditions. This alternative proof shares many similarities with that of Theorem 3, and it is made available in Appendix A for the interested reader.

\subsection{Computational Complexity}
\label{section:complexity}

\noindent
\textbf{Convex objective.}
Each call to the main algorithm $\textsc{MDA}(v,w)$ involves a recursive call to $\textsc{\textsc{MDA}}(v,u) $ and $\textsc{\textsc{MDA}}(u+1,w) $ with $u = \lfloor \frac{v+w}{2} \rfloor$, as well as
\begin{itemize}[nosep]
\item the solution of $\textsc{Rap}_{v,w}(L,R,\mathbf{\bar{c}},\mathbf{\bar{d}})$ for $L \in \{a_{v-1},b_{v-1}\}$ and $R \in \{a_w,b_w\}$;
\item up to four calls to the \textsc{Adjust} function;
\item a linear number of operations to set the bounds $\bar{c}_i$ and $\bar{d}_i$.
\end{itemize}
The function $\textsc{Adjust}$ uses a number of elementary operations which grows linearly with the number of variables. Moreover, in the presence of integer variables, each RAP subproblem with $n$ variables and bound $B  = R-L$ can be solved in $\cO(n \log \frac{B}{n})$ time using the algorithm of \cite{Frederickson1982} or \cite{Hochbaum1994}. As a consequence, the number of operations $\Phi(n,m,B)$ of MDA, as a function of the number of variables $n$ and constraints $m$, is bounded as
\begin{align*}
\Phi(n,m,B) &\leq \sum_{i=1}^{h} \left( K n +  \sum_{j=1}^{2^{h-i}} 
4 K' \left( \sigma[2^i j] - \sigma[2^i (j-1)] \right) \log{\left( \frac{B} {\sigma[2^i j] - \sigma[2^i (j-1)]} \right)} \right)  \\
&\leq K n h +  4 K' n h \log B,
\end{align*}
 where $K$ and $K'$ are constants and $h = 1 +  \lceil  \log_2 m  \rceil$.
Thus, $\Phi(n,m,B) \in \cO(n \log m \log B)$ in the integer case. For the continuous case, after scaling all problem parameters by $\lceil n/\epsilon \rceil$, the complexity of the algorithm for the search for an $\epsilon-$approximate solution becomes $\cO(n \log m \log \frac{nB}{\epsilon})$.\\

\noindent
\textbf{Quadratic and linear objectives.} More efficient RAP solution methods are known for specific forms of objective functions. The quadratic RAP with continuous variables, in particular, can be solved in $\cO(n)$ time \citep{Brucker1984}. For the integer case, reviewed in \cite{Katoh2013}, an $\cO(n)$ algorithm can be derived from Section 4.6 of \cite{Ibaraki1988}. Finally, in the linear case, each RAP subproblem can be solved in $\cO(n)$ time as a weighted median problem (see, e.g., \citealt{Korte2012a}). All these cases lead to $\cO(n \log m)$ algorithms for the corresponding RAP--NC. Note that no transformation or proximity theorem is needed for the continuous quadratic RAP, since the solutions of quadratic problems are representable.

\section{Computational Experiments}
\label{section experiments}

We perform computational experiments to evaluate the performance of the proposed algorithm in the presence of a linear objective, and for two convex objective functions arising in project crashing and speed optimization applications.
For linear problems, we compare with the network flow algorithm of \cite{Ahuja2008}, which achieved the previous best-known complexity of $\cO(n \log n)$ for the problem; this complexity is slightly improved to $\cO(n \log m)$ by the proposed MDA. For general convex objectives, no dedicated algorithm is available and we compare with the interior-point-based algorithm of MOSEK v7.1 for separable convex optimization. We finally report experimental analyses to evaluate the potential of this solver within a projected gradient method for the SVOREX problem (Section~\ref{Applications}), for ordinal regression. The algorithms are implemented in C++ and executed on a single core of a Xeon 3.07\,GHz CPU.
For accurate time measurements, any algorithm with a CPU time smaller than one second was executed multiple times in a loop (up to a total time of 10 seconds) to determine the average time of a run.

We generated benchmark instances with a number of variables $n \in \{10,20,50,100, 200,\dots,10^6\}$. Overall, $10$ random benchmark instances were produced for each problem size, leading to a total of 16$\times$10 instances with the same number of nested constraints as decision variables ($n=m$). For fine-grained complexity analyses in the case of the linear objective, we also removed random nested constraints to produce an additional set of 13$\times$10 instances with $m=100$ constraints and $n \in \{100,200,500,\dots,10^6\}$. 
For each instance, we generated the parameters $c_i$ and $d_i$ for $i \in \{1,\dots,n\}$ from uniform distributions in the range $[0.1,0.5]$ and $[0.5,0.9]$, respectively. Then, we defined two sequences of values $v_i$ and $w_i$, such that $v_0 = w_0 = 0$, $v_i = v_{i-1} + X^v_i$, and $w_i = w_{i-1} + X^w_i$ for $i \in \{1,\dots,n\}$, where $X^v_i$ and $X^w_i$ are random variables drawn from a uniform distribution in the range $[c_i,d_i]$. Finally, we set $a_i = \min \{v_i,w_i\}$ and $b_i = \max \{v_i,w_i\}$ for all $i$. We also selected a random parameter~$p_i$ in $[0,1]$ to characterize the objective function.
We conducted the experiments with four classes of objectives: a linear objective $\sum_{i=1}^n p_i x_i$, and three convex objectives defined as:
\begin{align}
[\text{F}]& \hspace*{1cm} f_i(x) = \frac{x^4}{4} +  p_i x,  \label{cost-F} \\
[\text{Crash}]& \hspace*{1cm} f_i(x) = k_i +  \frac{p_i}{x},  \label{cost-pert} \\
\text{and } [\text{Fuel}]& \hspace*{1cm} f_i(x) = p_i \times c_i \times \left(  \frac{c_i}{x} \right)^3, \label{vessel-fuel}
\end{align}
where the last two objectives are representative of applications in project crashing \citep{Foldes1993} and ship speed optimization~\citep{Ronen1982}.

\subsection{Linear Objective}
\label{linear}

We start the experimental analyses with the linear RAP--NC. We will refer to the network-flow-based approach of \cite{Ahuja2008} as ``FLOW'' in the text and tables. This method was precisely described in the original article, but no computational experiments or practical implementation were reported, so we had to implement it. The authors suggest the use of a red-black tree to locate the minimum-cost paths and a dynamic tree \citep{Tarjan1997,Tarjan2009} to manage the capacity constraints. This advanced data structure requires significant implementation effort and can result in high CPU time constants. We thus adopted a simpler structure, a segment tree \citep{Bentley1977} with lazy propagation, which allows evaluating and updating these capacities with the same complexity of $\cO(\log n)$ per operation (and possibly a higher speed in practice). The proposed MDA was implemented as in Algorithm \ref{algo-subprobSC}, solving each linear RAP subproblem in $\cO(n)$ time as a variant of a weighted median problem \citep{Korte2012a}.

We executed both algorithms on each instance.
The results for the instances with as many nested constraints as decision variables ($n=m$) are reported in Figure \ref{boxplot-time-n}. To evaluate the growth of the computational effort of the algorithms as a function of problem size, we fitted the computational time as a power law $f(n) = \alpha \times n^\beta$ of the number of variables $n$, via a least-squares regression of an affine function on the log-log graph (left figure). We also display as boxplots the ratio of the computational time of MDA and FLOW (right figure). The same conventions are used to display the results of the experimental analyses with a fixed number of constraints ($m=100$) and increasing number of variables $n$ in Figure \ref{boxplot-time-m}. Finally, the detailed average computational times for each group of 10 instances are reported in Table \ref{detailed-n}.

\begin{figure}[htbp]
\hspace*{-0.5cm}
\begin{minipage}[b]{7cm}
\hspace*{0.3cm} \begin{scriptsize}Time(s)\end{scriptsize}

\includegraphics[height=5.42cm]{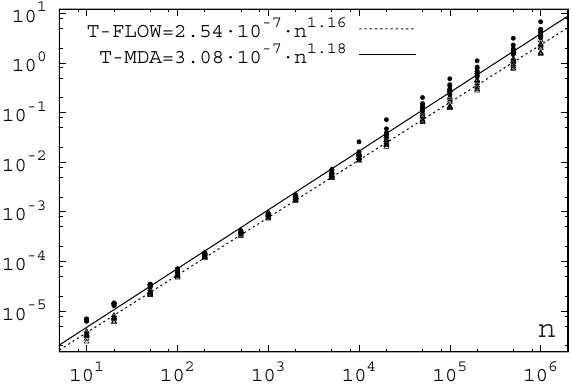}
\end{minipage} \hspace*{1.5cm}
\begin{minipage}[b]{7cm}
\begin{scriptsize}$T_\textsc{FLOW}/T_\textsc{MDA}$\end{scriptsize}

\includegraphics[height=5.42cm]{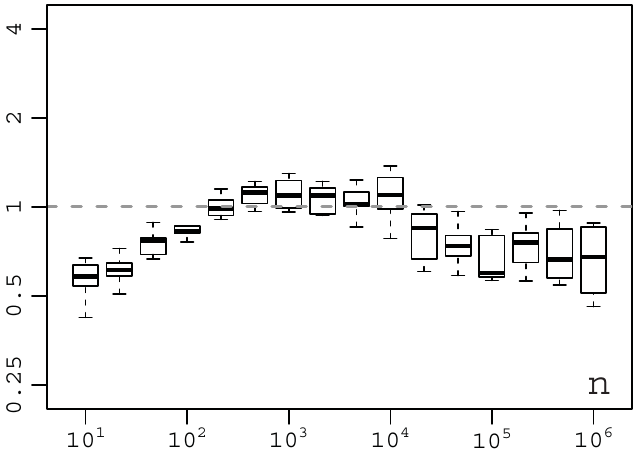}
\end{minipage}
\vspace*{-0.05cm}
\caption{Varying $n \in \{10,\dots,10^6\}$ and $m=n$. Left figure: CPU time of both methods as $n$ and $m$ grow. Right figure: Boxplots of the ratio $T_\textsc{FLOW}/T_\textsc{MDA}$.}
\label{boxplot-time-n}
\end{figure}

\begin{figure}[htbp]
\vspace*{-0.7cm}
\hspace*{-0.5cm}
\begin{minipage}[b]{7cm}
\hspace*{0.3cm} \begin{scriptsize}Time(s)\end{scriptsize}

\includegraphics[height=5.43cm]{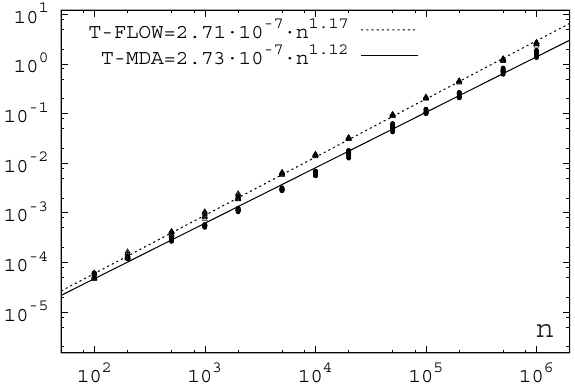}
\end{minipage} \hspace*{1.5cm}
\begin{minipage}[b]{7cm}
\begin{scriptsize}$T_\textsc{FLOW}/T_\textsc{MDA}$\end{scriptsize}

\includegraphics[height=5.42cm]{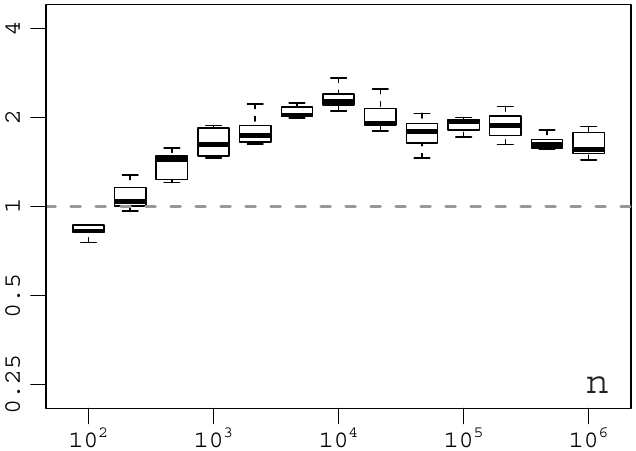}
\end{minipage}
\caption{Linear Objective. Varying $n \in \{10,\dots,10^6\}$ and fixed $m=100$. Left figure: CPU time of both methods as $n$ grows. Right figure: Boxplots of the ratio $T_\textsc{FLOW}/T_\textsc{MDA}$.}
\label{boxplot-time-m}
\end{figure}

\begingroup
\renewcommand{\arraystretch}{1.2}
\begin{table}[htbp]
\begin{center}
\caption{Detailed average CPU times for experiments with a linear objective}
\label{detailed-n}
\scalebox{0.78}
{
\setlength{\tabcolsep}{12pt}
\hspace*{-0.4cm}
\begin{tabular}{|ll|c|c| c |ll|c|c|}
\cline{1-4} \cline{6-9}
\multicolumn{2}{|c|}{\strut\vspace*{-0.1cm}}&\multicolumn{2}{c|}{\strut\vspace*{-0.1cm}}&&\multicolumn{2}{c|}{\strut\vspace*{-0.1cm}}&\multicolumn{2}{c|}{\strut\vspace*{-0.1cm}}\\
\multicolumn{2}{|c|}{\textbf{Variable} $\mathbf{m}$}&\multicolumn{2}{c|}{\textbf{CPU Time(s)}}&&\multicolumn{2}{c|}{\textbf{Fixed} $\mathbf{m}$}&\multicolumn{2}{c|}{\textbf{CPU Time(s)}} \\
$n$&$m$&FLOW&MDA&& $n$&$m$&FLOW&MDA \\
\cline{1-4} \cline{6-9}
\multicolumn{2}{|c|}{\strut\vspace*{-0.2cm}}&&&&\multicolumn{2}{c|}{\strut\vspace*{-0.2cm}}&&\\
10&10&\num{0.000002753}&\num{0.000004784}&&100&100&\num{0.000050934}&\num{0.000059476}\\
20&20&\num{0.000006262}&\num{0.000010242}&&200&100&\num{0.000135978}&\num{0.000126415}\\
50&50&\num{0.000021454}&\num{0.000028513}&&500&100&\num{0.000394375}&\num{0.000285976}\\
100&100&\num{0.000050576}&\num{0.000058915}&&1000&100&\num{0.000907237}&\num{0.000551754}\\
200&200&\num{0.000126376}&\num{0.000125758}&&2000&100&\num{0.002065989}&\num{0.001139486}\\
500&500&\num{0.000371994}&\num{0.00033622}&&5000&100&\num{0.006162977}&\num{0.002958873}\\
1000&1000&\num{0.000842986}&\num{0.000757082}&&10000&100&\num{0.0144292}&\num{0.006262521}\\
2000&2000&\num{0.001865427}&\num{0.001740006}&&20000&100&\num{0.031677652}&\num{0.01571567}\\
5000&5000&\num{0.005426528}&\num{0.005199759}&&50000&100&\num{0.092668418}&\num{0.052614627}\\
10000&10000&\num{0.012305958}&\num{0.011198173}&&100000&100&\num{0.203549512}&\num{0.108438407}\\
20000&20000&\num{0.026222298}&\num{0.032138957}&&200000&100&\num{0.441196771}&\num{0.235623163}\\
50000&50000&\num{0.079430967}&\num{0.105266726}&&500000&100&\num{1.19560909}&\num{0.718873748}\\
100000&100000&\num{0.151968169}&\num{0.225724083}&&1000000&100&\num{2.564233333}&\num{1.598551587}\\
\cline{6-9}
200000&200000&\num{0.36655844}&\num{0.485572598}\\
500000&500000&\num{0.967679801}&\num{1.369780555}\\
1000000&1000000&\num{1.989558333}&\num{2.9783}\\
\cline{1-4}
\end{tabular}
}
\end{center}
\end{table}
\endgroup

From these experiments, we observe that the computational times of the two methods are very similar in terms of magnitude and growth rate. When $m=n$, the algorithms have the same theoretical complexity of $\cO(n \log n)$, as confirmed by the power law regression, with an observed growth that is close to linear (in $n^{1.16}$ and $n^{1.18}$). The FLOW algorithm is on average $1.1\times$ to $1.4\times$ faster than MDA for $n \in [10,100] \cup [10^4,10^6]$, for instances with the same number of variables and constraints. On the other hand, MDA is on average $2\times$ faster than FLOW when $m$ is fixed and $n$ grows beyond~$1000$. This is due to the difference in computational complexity: $\cO(n \log m)$ for MDA instead of $\cO(n \log n)$. MDA and FLOW solve the largest instances, with up to $n=m=10^6$ constraints and variables, in three and two seconds on average, respectively.

Overall, the two algorithms have similar performance for linear objectives, and the CPU differences are small. Since these algorithms are based on drastically different principles, they lead the way to different methodological extensions. The computational complexity of FLOW is tied to its efficient use of a dynamic tree data structure, while the complexity of the MDA stems from its ``monotonic'' divide-and-conquer strategy. Because of this structure, MDA should be a good choice for re-optimization after a change of a few parameters, as well as for the iterative solution of multiple RAP--NC, e.g., for speed optimization within an algorithm enumerating a large number of similar visit sequences, since it can reuse the solutions of smaller subproblems (see, e.g., \citealt{Norstad2010} and \citealt{Vidal2012b,Vidal2015b}).

\subsection{Separable Convex Objectives} 
\label{convex}

In contrast with the case of linear objective functions, no specialized algorithm has been designed for the RAP--NC with separable convex costs to date. To illustrate the possible gain achieved by the use of a dedicated algorithm rather than a general-purpose solver, we do a simple comparison of the CPU time of MDA to that of the MOSEK v7.1 solver on two sets of instances with objective functions derived from project crashing and speed optimization applications. MOSEK is based on an interior-point method and is a good representative of the current generation of separable convex optimization solvers. We set a time limit of one hour. To simplify the execution of these experiments, we use a binary search over the single dual variable to solve each continuous RAP subproblem \citep{Patriksson2008} within a precision of $10^{-9}$. As these approximations may stack up $\log m$ times in the recursion, we obtain an overall good accuracy but do not guarantee $\epsilon$-proximity~in~these~tests.

\begin{figure}[htbp]
\hspace*{-0.5cm}

\begin{minipage}[b]{7cm}
\hspace*{0.3cm} \begin{scriptsize}Time(s)\end{scriptsize} \hspace*{2.5cm}[F]

\includegraphics[height=5.35cm]{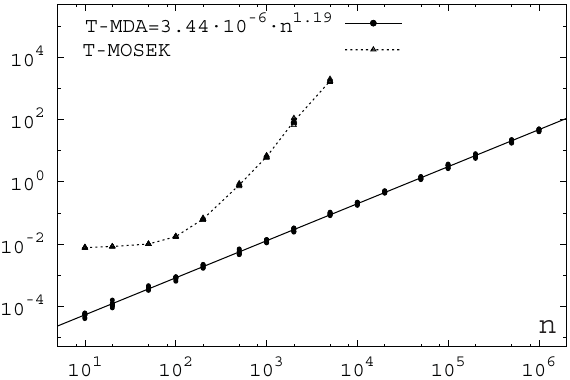}
\end{minipage} \hspace*{1.4cm}
\begin{minipage}[b]{7cm}
\hspace*{0.3cm} \begin{scriptsize}Time(s)\end{scriptsize} \hspace*{2.2cm}[Crash]

\includegraphics[height=5.35cm]{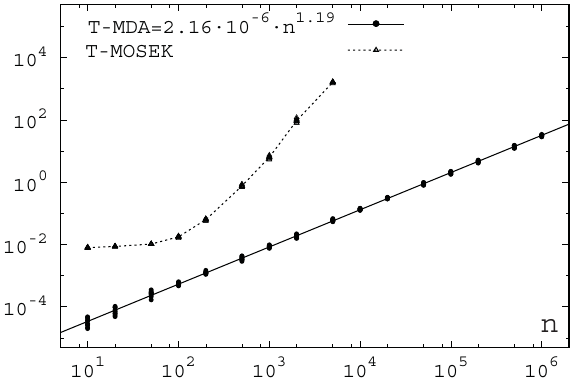}
\end{minipage}
\vspace*{0.25cm}

\begin{minipage}[b]{7cm}
\hspace*{0.3cm} \begin{scriptsize}Time(s)\end{scriptsize} \hspace*{2.2cm}[Fuel]

\includegraphics[height=5.35cm]{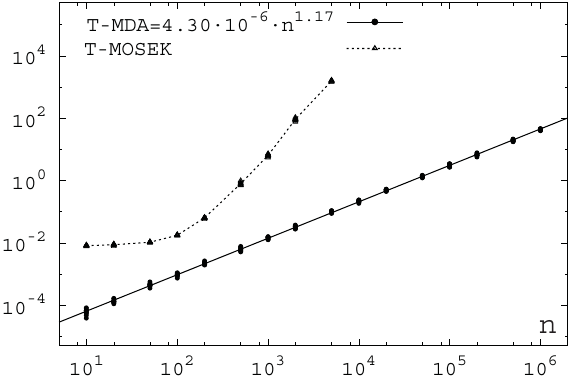}
\end{minipage} \hspace*{1.72cm}
\begin{minipage}[b]{7cm}
\begin{scriptsize}$T_\textsc{Mosek}/T_\textsc{MDA}$\end{scriptsize}

\includegraphics[height=5.5cm]{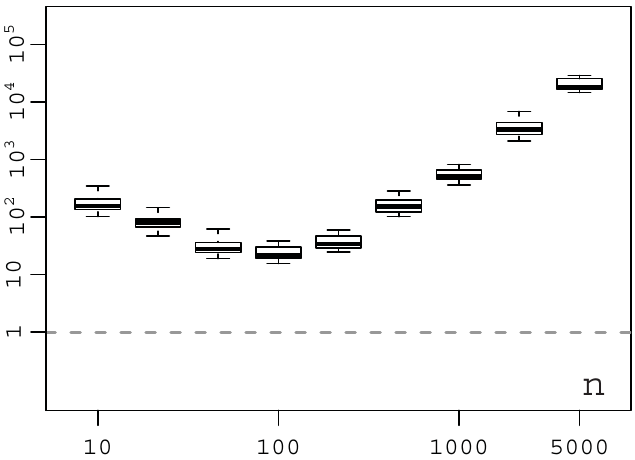}
\end{minipage}
\vspace*{0.2cm}
\caption{Convex Objective. From left to right and top to bottom: CPU time of both methods as $n$ grows and $m=n$ for the objectives [F], [Crash], and [Fuel]. Bottom right figure: Boxplots of the ratio $T_\textsc{Mosek}/T_\textsc{MDA}$.}
\label{boxplot-time-convex}
\end{figure}

\begingroup
\setlength{\medmuskip}{0mu}
\renewcommand{\arraystretch}{1.2}
\begin{table}[htbp]
\begin{center}
\caption{Detailed average CPU-time for experiments with a separable convex objective}
\label{detailed-n-convex}
\scalebox{0.8}
{
\setlength{\tabcolsep}{12pt}
\begin{tabular}{|ll|ccc|ccc|}
\hline
&&\multicolumn{3}{c|}{\vspace*{-0.2cm}}&\multicolumn{3}{c|}{\vspace*{-0.2cm}} \\
&&\multicolumn{3}{c|}{\textbf{CPU Time(s) -- MDA}}&\multicolumn{3}{c|}{\textbf{CPU Time(s) -- MOSEK}} \\
&&\multicolumn{3}{c|}{\vspace*{-0.25cm}}&\multicolumn{3}{c|}{\vspace*{-0.25cm}} \\
$n$&$m$&[F]&[Crash]&[Fuel]&[F]&[Crash]&[Fuel] \\
&&\multicolumn{3}{c|}{\vspace*{-0.25cm}}&\multicolumn{3}{c|}{\vspace*{-0.25cm}} \\
\hline
&&\multicolumn{3}{c|}{\vspace*{-0.2cm}}&\multicolumn{3}{c|}{\vspace*{-0.2cm}} \\
10&10&\num{0.000052804}&\num{0.000032711}&\num{0.000061129}&\num{0.007686943}&\num{0.007832116}&\num{0.008061434}\\
20&20&\num{0.000114081}&\num{0.000073175}&\num{0.000133301}&\num{0.008274424}&\num{0.008604278}&\num{0.008640814}\\
50&50&\num{0.000380328}&\num{0.000263063}&\num{0.000445094}&\num{0.009952053}&\num{0.010301065}&\num{0.010400973}\\
100&100&\num{0.000803668}&\num{0.000538618}&\num{0.000930292}&\num{0.017299139}&\num{0.017526892}&\num{0.01740955}\\
200&200&\num{0.001934188}&\num{0.001231371}&\num{0.002160542}&\num{0.063120002}&\num{0.062191123}&\num{0.062980577}\\
500&500&\num{0.005448312}&\num{0.003550539}&\num{0.006211667}&\num{0.778735352}&\num{0.756403467}&\num{0.785640721}\\
1000&1000&\num{0.012697424}&\num{0.008608136}&\num{0.014253454}&\num{6.314833334}&\num{6.293333333}&\num{6.368833333}\\
2000&2000&\num{0.028767787}&\num{0.018728026}&\num{0.031944613}&\num{85.749}&\num{93.761}&\num{90.474}\\
5000&5000&\num{0.092704885}&\num{0.060521584}&\num{0.098564133}&\num{1703.926}&\num{1605.518}&\num{1554.495}\\
10000&10000&\num{0.201499218}&\num{0.133855399}&\num{0.212543624}&---&---&---\\
20000&20000&\num{0.469045625}&\num{0.304190384}&\num{0.482383221}&---&---&---\\
50000&50000&\num{1.305475757}&\num{0.873814239}&\num{1.327877777}&---&---&---\\
100000&100000&\num{3.1194}&\num{2.023380952}&\num{3.06865}&---&---&---\\
200000&200000&\num{6.683166667}&\num{4.584}&\num{6.613666667}&---&---&---\\
500000&500000&\num{19.838}&\num{13.538}&\num{19.101}&---&---&---\\
1000000&1000000&\num{45.403}&\num{31.004}&\num{42.991}&---&---&---\\
\hline
\end{tabular}
}
\end{center}
\end{table}
\endgroup 

The results are reported in Figure \ref{boxplot-time-convex} and Table \ref{detailed-n-convex}. In the figure, the power-law regressions are presented only for MDA, since MOSEK does not exhibit polynomial behavior, likely due to the computational effort related to the initialization of the solver for  small problems. In contrast, the computational time of MDA grows steadily in $\cO(n^{1.19})$ at most. This observation is consistent with the theoretical $\cO(n \log m)$ complexity of the method.
Within one hour, MOSEK solved all the instances up to $n = 5,000$ decision variables. In contrast, MDA solved all the available instances with up to one million variables. For the largest benchmark instances, the CPU time of the method did not exceed $50$ seconds.
As illustrated in the bottom-right subfigure, the ratio of the CPU time of MOSEK and MDA ranges between 16 and 28,000. For all the instances, significant CPU time is saved when using the monotonic divide-and-conquer algorithm instead of a general-purpose~solver.

\subsection{Non-Separable Convex Objective -- Support Vector Ordinal Regression} 
\label{ordinal}

Our last experimental analysis is concerned with the SVOREX model, presented in Section~\ref{Applications}. It is a non-separable convex optimization problem over a special case of the RAP--NC constraint polytope. The current state-of-the-art algorithm for this problem, proposed by \cite{Chu2007a}, is based on a working-set decomposition. Iteratively, a set of variables is selected to be optimized over, while the others remain fixed. This approach leads to a (non-separable) restricted problem with fewer variables which can be solved to optimality.
The authors rely on a \emph{minimal working set} containing the two variables which most violate the KKT conditions (see \citealt{Chu2007a}, pp. 799--800, for all equations involved).

The advantage of a minimal working set comes from the availability of analytical solutions for the restricted problems.
On the other hand, larger working sets can be beneficial in order to reduce the number of iterations until convergence (see, e.g., \citealt{Joachims1999}). However this would require an efficient method for the resolution of the reduced problems. This is how the proposed RAP--NC solver can provide a meaningful option along this direction. In order to evaluate such a proof of concept, we conduct a simple experiment which consists of generating larger working sets within the approach of \cite{Chu2007a} and solving the resulting reduced problems with the help of the RAP--NC algorithm. As these reduced problems are non-separable convex, the RAP--NC algorithm is being used for the projection steps within a projected gradient descent procedure. The overall solution approach is summarized in Algorithm \ref{algo:WSB}, in which $\mathcal{W}$ is the working set, $z$ is the objective function, and $\gamma$ is the fixed step size of the gradient descent.

\begin{algorithm}[htbp]
\linespread{1.2}\selectfont

$\boldsymbol\alpha = \boldsymbol\alpha^* = \mathbf{0}$ \tcp*{Initial Solution set to 0}

\While{there exists samples that violate the KKT conditions}
{
Select a working set $\mathcal{W}$ of maximum size $n_\textsc{ws}$ 

\For{$n_\textsc{grad}$ iterations}
{

 \tcp{Take a step}

\For{$j \in \{1,\dots,r\}$ and $i \in \{1,\dots,n^j\}$}
{

$\hat{\alpha}_i^{j}  = 
\begin{cases}
 \alpha_i^{j} + \gamma  \frac{\partial z}{\partial \alpha_i^{j} } & \text{if } (i,j)  \in \mathcal{W}  \\
 \alpha_i^{j} &  \text{otherwise}
\end{cases}$ \hspace*{0.2cm} ; \hspace*{0.2cm}
$\hat{\alpha}_i^{*j}  = 
\begin{cases}
 \alpha_i^{*j} + \gamma  \frac{\partial z}{\partial \alpha_i^{*j} } &\text{if } (i,j)  \in \mathcal{W}  \\
 \alpha_i^{*j} &  \text{otherwise}
\end{cases}$
}

 \tcp{Solve the projection subproblem as a RAP-NC}

\vspace*{-0.5cm}

\begin{equation*}
(\boldsymbol\alpha,\boldsymbol\alpha^*) \gets 
\begin{cases}
\min\limits_{\boldsymbol\alpha,\boldsymbol\alpha^*} & \sum\limits_{(i,j) \in \mathcal{W}}  \left( (\alpha_i^j -  \hat{\alpha}_i^j )^2 + (\alpha_i^{*j} -  \hat{\alpha}_i^{*j} )^2 \right) \\
\text{ s.t.} & \text{Equations (\ref{SVOR:2})--(\ref{SVOR:5})} \\
& \alpha_i^j  = \hat{\alpha}_i^j \text{ and }   \alpha_i^{*j}  = \hat{\alpha}_i^{*j} \text{  \hspace*{0.15cm} if } (i,j) \notin \mathcal{W} \\
\end{cases}
\end{equation*}
}
}

 \caption{Solving SVOREX via RAP-NC subproblems} \label{algo2} 
 \label{algo:WSB} 
\end{algorithm}

To obtain a larger working set,
we repeatedly select the most-violated sample pair
until either reaching the desired size or not finding any remaining violation.
In our experiments, we consider working sets of size $n_\textsc{ws} \in \{2,4,6,10\}$, a step size of $\gamma = 0.2$ and $n_\textsc{grad} = 20$ iterations for the projected gradient descent.
We use the eight problem instances introduced in \cite{Chu2007a}, with the same Gaussian kernel, penalty parameter, and guidelines for data preparation (normalizing the input vectors to zero mean and unit variance, and using equal-frequency binning to discretize the target values into five ordinal scales).

Table \ref{Results-SVOREX} gives the results of these experiments. The columns report, in turn, the problem instance name, its number of samples $N$, the dimension $D$ of its feature space, and characteristics of the optimal solutions: the number of variables set to~$0$ (correct classification), to $C$ (misclassified), and to intermediate values (support vectors). For each working-set size $n_\textsc{ws}$, the total number of working set selections $I_\textsc{ws}$ done by the algorithm is also presented, as well as the CPU time in seconds.
The fastest algorithm version is underlined for each instance.

\begingroup
\renewcommand{\arraystretch}{1.1}
\begin{table}[!h]
\begin{center}
\caption{SVOREX resolution -- impact of the working-set size and solution features}
\label{Results-SVOREX}
\scalebox{0.75}
{
\setlength{\tabcolsep}{10pt}
\begin{tabular}{|lcc|ccc|rrr|}
\multicolumn{9}{c}{\vspace*{-0.45cm}}\\
\hline
&&&&&&&&\vspace*{-0.35cm}\\
\multirow{2}{*}{\textbf{Instance}}&\multirow{2}{*}{$\mathbf{N}$}&\multirow{2}{*}{$\mathbf{D}$}&\multicolumn{3}{c|}{\textbf{Solution Variables s.t.}}&\multirow{2}{*}{$\mathbf{n_\textsc{ws}}$}&\multirow{2}{*}{$\mathbf{I_{ws}}$}&\multirow{2}{*}{\textbf{T(s)}}\\
&&&$\alpha = 0$&$\alpha = C$&$\alpha \in ]0,C[$&&&\\
&&&&&&&&\vspace*{-0.5cm}\\
\hline
&&&&&&&&\vspace*{-0.5cm}\\
\multirow{4}{*}{Abalone} &\multirow{4}{*}{1000}&\multirow{4}{*}{8}&\multirow{4}{*}{39\%}&\multirow{4}{*}{32\%}&\multirow{4}{*}{29\%}&\textbf{\underline{2}}&\textbf{\underline{118233}}&\textbf{\underline{13.46}}\\
&&&&&&4&96673&21.51\\
&&&&&&6&78433&26.34\\
&&&&&&10&60605&35.46\\
&&&&&&&&\vspace*{-0.5cm}\\
\hline
&&&&&&&&\vspace*{-0.5cm}\\
\multirow{4}{*}{Bank}&\multirow{4}{*}{3000}&\multirow{4}{*}{32}&\multirow{4}{*}{25\%}&\multirow{4}{*}{0\%}&\multirow{4}{*}{75\%}&2&139468&68.41\\
&&&&&&4&52073&63.02\\
&&&&&&\textbf{\underline{6}}&\textbf{\underline{31452}}&\textbf{\underline{45.22}}\\
&&&&&&10&21310&47.66\\
&&&&&&&&\vspace*{-0.5cm}\\
\hline
&&&&&&&&\vspace*{-0.5cm}\\
\multirow{4}{*}{Boston}&\multirow{4}{*}{300}&\multirow{4}{*}{13}&\multirow{4}{*}{41\%}&\multirow{4}{*}{0\%}&\multirow{4}{*}{59\%}&2&7207&0.43\\
&&&&&&\textbf{\underline{4}}&\textbf{\underline{3697}}&\textbf{\underline{0.40}}\\
&&&&&&6&2840&0.46\\
&&&&&&10&2076&0.54\\
&&&&&&&&\vspace*{-0.5cm}\\
\hline
&&&&&&&&\vspace*{-0.4cm}\\
\multirow{4}{*}{California} &\multirow{4}{*}{5000}&\multirow{4}{*}{8}&\multirow{4}{*}{51\%}&\multirow{4}{*}{43\%}&\multirow{4}{*}{6\%}&\textbf{\underline{2}}&\textbf{\underline{250720}}&\textbf{\underline{124.46}}\\
&&&&&&4&189289&185.79\\
&&&&&&6&166879&245.08\\
&&&&&&10&146170&360.52\\
&&&&&&&&\vspace*{-0.5cm}\\
\hline
&&&&&&&&\vspace*{-0.5cm}\\
\multirow{4}{*}{Census} &\multirow{4}{*}{6000}&\multirow{4}{*}{16}&\multirow{4}{*}{38\%}&\multirow{4}{*}{4\%}&\multirow{4}{*}{59\%}&\textbf{\underline{2}}&\textbf{\underline{349894}}&\textbf{\underline{242.11}}\\
&&&&&&4&206951&301.74\\
&&&&&&6&180608&393.28\\
&&&&&&10&155731&574.28\\
&&&&&&&&\vspace*{-0.5cm}\\
\hline
&&&&&&&&\vspace*{-0.5cm}\\
\multirow{4}{*}{Computer}&\multirow{4}{*}{4000}&\multirow{4}{*}{21}&\multirow{4}{*}{64\%}&\multirow{4}{*}{32\%}&\multirow{4}{*}{4\%}&2&290207&168.94\\
&&&&&&4&140270&161.45\\
&&&&&&\textbf{\underline{6}}&\textbf{\underline{98948}}&\textbf{\underline{153.56}}\\
&&&&&&10&68616&193.10\\
&&&&&&&&\vspace*{-0.5cm}\\
\hline
&&&&&&&&\vspace*{-0.5cm}\\
\multirow{4}{*}{Machine CPU}&\multirow{4}{*}{150}&\multirow{4}{*}{6}&\multirow{4}{*}{49\%}&\multirow{4}{*}{9\%}&\multirow{4}{*}{41\%}&2&28856&1.24\\
&&&&&&\textbf{\underline{4}}&\textbf{\underline{11534}}&\textbf{\underline{0.86}}\\
&&&&&&6&8144&0.91\\
&&&&&&10&6363&1.24\\
&&&&&&&&\vspace*{-0.5cm}\\
\hline
&&&&&&&&\vspace*{-0.5cm}\\
\multirow{4}{*}{Pyrimidines}&\multirow{4}{*}{50}&\multirow{4}{*}{27}&\multirow{4}{*}{21\%}&\multirow{4}{*}{0\%}&\multirow{4}{*}{79\%}&2&935&0.035\\
&&&&&&4&367&0.021\\
&&&&&&\textbf{\underline{6}}&\textbf{\underline{218}}&\textbf{\underline{0.018}}\\
&&&&&&10&144&0.023\\
\hline
\end{tabular}
}
\end{center}
\end{table}
\endgroup

As measured in these experiments, the CPU time of the algorithms ranges between 0.018 seconds for the smallest problem instances (with 50 samples and 27 dimensions) and 574.28 seconds for the largest case (6000 samples and 16 dimensions). The size of the working set has a significant impact on the number of iterations of the method and its CPU time.

In all cases, the number of iterations decreases significantly when the size of the working set grows. In terms of CPU time, the fastest results are either achieved with a working set of size two or six (with three instances in each case). We observe that the three instances for which a larger working set contributed to CPU-time reductions are those with higher-dimension feature spaces (dimension 21 to 32). For these instances, using a larger working set helped reduce the CPU time by a factor of $1.1$ to $1.9$ as compared to using a two-samples working set. In comparison, \cite{Joachims1999} reported a speedup of $1.5$ to $2.0$ when using ten-samples working sets for SVM, which can be viewed as a special case of SVOREX with two classes.

To achieve a gain in CPU time, the number of iterations should decrease more than linearly as a function of the working set size. This is due to the effort spent updating the gradient, necessary for the verification of the KKT conditions and the working-set selection, which grows linearly with the product $I_\textsc{ws} \times N \times n_\textsc{ws}$ (using efficient incremental updates), and which remains a major bottleneck for SVOREX and SVM algorithms (see, e.g., the discussions in \citealt{Joachims1999} and \citealt{Chang2013}). Usually, instances with a feature space of high dimension exhibit a fast decrease in the number of iterations as a function of the working set size, as their solutions include a larger proportion of variables taking values in $(0,C)$  (support vectors), values which are more quickly reached via simultaneous optimizations of several variables. As such, larger working sets are likely to be more useful in feature spaces of higher dimension.

Moreover, future research avenues concern possible improvements of the algorithm (e.g., using shrinking or a double-loop scheme -- \citealt{Keerthi2001}), adaptive choices of working-set size based on analyses of the structure of the data set and solutions, as well as more advanced selection rules, e.g., based on Zoutendijk's descent direction \citep{Joachims1999} or second order information \citep{Lin2005}. These options for improvement are now possible due to the availability of a fast algorithm for the resolution of the restricted problems.

\section{Concluding Remarks}
\label{sec:conclusions}

In this article, we have highlighted the importance of the RAP--NC, which is a problem connected with a wide range of applications in production and transportation optimization, portfolio management, sampling optimization, telecommunications and machine learning.
To solve this problem, we proposed a decomposition algorithm, based on monotonicity principles coupled with divide-and-conquer, leading to new complexity breakthroughs for a variety of objectives (linear, quadratic, and convex), with continuous or integer variables, and to the first known strongly polynomial algorithm for the quadratic integer RAP--NC. In terms of practical performance, the algorithm matches the best dedicated (flow-based) algorithm for the linear case, outperforms general-purpose solvers by several orders of magnitude in the convex case, and opens interesting perspectives of algorithmic improvements for the SVOREX problem in machine learning.

The algorithm can be seen as a generalization of the method of \cite{Vidal2014a}, with some fundamental differences related to the number and the nature of the subproblems. It is not based on classical greedy steps and scaling, or on flow propagation techniques, often exploited for this problem family. It is an important research question to see how far this decomposition technique can be generalized. In particular, the approach can very likely be extended to the resource allocation problem with a TREE of lower \emph{and upper} constraints \citep{Hochbaum1994}. Other optimization problems related to PERT (Program Evaluation and Review Technique) may exhibit monotonicity properties as a function of time constraints or budget bounds, and we should investigate how to decompose efficiently their variables and constraints while maintaining a low computational complexity. Similarly, extended formulations involving the intersection of two or more RAP--NC type of constraint polytopes deserve a closer look. These are all open important research directions which can be explored in the near future.

\section*{Acknowledgments}

This research was partially supported by the National Counsel of Technological and Scientific Development and \emph{Funda\c{c}\~ao de Amparo \`a Pesquisa do Estado do Rio de Janeiro} (FAPERJ) in Brazil, grants 308498/2015-1 and E-26/203.310/2016, and by the Office of Naval Research (ONR), USA, grant N00014-15-1-2083.

\section*{Appendix A -- Proof of Theorem \ref{proximity theorem}, based on KKT conditions}

\noindent
\proof
\begin{leftbar}
The proof shares many similarities with that of Theorem \ref{p1}. It exploits the fact that an integer solution $\mathbf{x^*}$ of the $\text{RAP--NC}^\textsc{pl}$ is also an optimal solution of the continuous problem~(Theorem~\ref{theo:reformulation}) and thus satisfies the KKT conditions of Equations (\ref{KKT3})--(\ref{KKT5}) based on the functions~$f_i^\textsc{pl}$. We first state two lemmas, that will be used later to link the values of the subderivatives of $f_i$ and $f_i^\textsc{pl}$.

\begin{lemma}
\label{lemma1}
Consider $y \in \mathbb{R}$ and $x \in \mathbb{Z}$ such that $y + 1\leq x$, and a convex function $f$. If $\phi_y \in \partial f(y)$ and $\phi_x \in \partial f^\textsc{pl}(x)$, then $\phi_y \leq \phi_x$.
\end{lemma}

\noindent
\emph{Proof of Lemma \ref{lemma1}.}
By definition,
$f(x) - f(y) \geq \phi_y (x - y)$ and 
$f^\textsc{pl}(x-1) - f^\textsc{pl}(x) \geq \phi_x (x - 1 - x) = - \phi_x$.
Moreover, $y \leq x-1 \leq x$, $f$ is convex, and $f$ coincides with $f^\textsc{pl}$ at $x$ and $x-1$, so
$$\phi_y \leq \frac{f(x) - f(y)}{x-y} \leq \frac{f(x) - f(x-1)}{x- (x -1)} =  f^\textsc{pl}(x) - f^\textsc{pl}(x-1) \leq \phi_x.$$

\begin{lemma}
\label{lemma2}
Consider $y \in \mathbb{Z}$ and $x \in \mathbb{R}$ such that $y + 1\leq x$, and a convex function $f$. If $\phi_y \in \partial f^\textsc{pl}(y)$ and $\phi_x \in \partial f(x)$, then $\phi_y \leq \phi_x$.
\end{lemma}

\noindent
\emph{Proof of Lemma \ref{lemma2}.}
By definition,
$f^\textsc{pl}(y+1) - f^\textsc{pl}(y) \geq \phi_y $ and 
$f(y) - f(x) \geq \phi_x (y - x)$.
Moreover, $y \leq y+1 \leq x$, $f$ is convex and $f$ coincides with $f^\textsc{pl}$ at $y$ and $y+1$, so
$$\phi_y \leq f^\textsc{pl}(y+1) - f^\textsc{pl}(y) =   \frac{f(y+1) - f(y)}{(y+1) - y} \leq  \frac{f^\textsc{pl}(x) - f^\textsc{pl}(y)}{x-y}  \leq \phi_x.$$

The main proof follows.
Let $\mathbf{x}$ be an optimal continuous solution of the RAP--NC. If Equation (\ref{eqprox}) is satisfied, then the proof is complete; otherwise there exists $s \in \{1, \dots, n\}$ such that $|x_s - x^*_s | \geq n - 1$. 
We consider here the case where $x_s \geq x^*_s + n - 1$, the other case being symmetric. 
Let~$r$ be the greatest index in $\{1,\dots,s\}$ such that $\sum_{k=1}^{r-1} x_k \geq \sum_{k=1}^{r-1} x^*_k$, and~$t$ be the smallest index in $\{s,\dots,n\}$ such that $\smash{\sum_{k=1}^{t} x_k \leq \sum_{k=1}^{t} x^*_k}$.
By the definition of $r$~and~$t$, it follows that $\smash{\sum_{i=r}^t x^*_i \geq \sum_{i=r}^t x_i}$, and thus $\smash{\sum_{i\in\{r,\dots,t\}-s} x^*_i \geq \sum_{i\in\{r,\dots,t\}-s} x_i + n - 1}$. Since $|\{r,\dots,t\}-s| \leq n-1$, there exists $u \in \{r,\dots,t\} - s$ such that $x^*_u \geq x_u + 1$. 

Two cases can arise:

If $u < s$, for each $j$ such that $\sigma[j] \in \{ u,\dots, s-1\}$, 
$\bar{a}_j \leq \sum_{k=1}^{\sigma[j]} x_k < \sum_{k=1}^{\sigma[j]} x^{*}_k \leq \bar{b}_j$  and thus \mbox{$\kappa^*_j = \lambda_j = 0$}. As a consequence, $\Phi_{i} \geq \Phi_{i+1}$ and $\Phi^{*}_{i} \leq \Phi^{*}_{i+1}$  for \mbox{$i \in \{u,\dots,s-1\}$}.

If $u > s$, for each  $j$ such that $\sigma[j]   \in \{ s,\dots, u-1\}$, 
$\smash{\bar{a}_j \leq \sum_{k=1}^{\sigma[j]} x^{*}_k < \sum_{k=1}^{\sigma[j]} x_k \leq \bar{b}_j}$ and thus $\lambda^*_j = \kappa_j = 0$. As a consequence, $\Phi_{i} \leq \Phi_{i+1}$ and $\Phi^{*}_{i} \geq \Phi^{*}_{i+1}$  for $i \in \{s,\dots,u-1\}$.

Moreover, $\{x^*_s + n - 1 \leq x_s,  \Phi^*_{s} \in \partial f_s(x^*_s),  \Phi_{s} \in \partial f^{\textsc{pl}}_s(x_s) \} \Rightarrow \Phi^*_{s} \leq \Phi_{s}$ (Lemma~\ref{lemma1}), and
 $\{x_u + 1 \leq x^*_u,   \Phi_{u} \in \partial f^{\textsc{pl}}_u(x_u) ,  \Phi^*_{u} \in \partial f_u(x^*_u)\} \Rightarrow \Phi_{u} \leq \Phi^*_{u}$ (Lemma~\ref{lemma2}).
Combining all the relations leads to 
$
\Phi_s \leq \Phi_u \leq \Phi^{*}_u \leq  \Phi^{*}_s \leq \Phi_s, 
$
and thus there exists $\Psi \in \Re$ such that $\Phi^*_i = \Phi_i = \Psi$ for \mbox{$i \in \{u,\dots,s\}$} if $u < s$ (or \mbox{$i \in \{s,\dots,u\}$} if $s < u$). As in the proof of Theorem \ref{p1}, this implies that 
the functions $f_s$ and $f_u$ are affine with slope $\Psi$ over $[x_s^\textsc{min},x_s^\textsc{max}]$ and $[x_u^\textsc{min},x_u^\textsc{max}]$, respectively, where $x_i^\textsc{min} = \min\{x_i,x^*_i\}$. Observe that the new solution $\mathbf{x'} = \mathbf{x} - \mathbf{e}^s + \mathbf{e}^u$ is a feasible solution with the same cost as $\mathbf{x}$, hence optimal.
Moreover, we note that $\sum_{i=1}^n  \max\{| x'_i - x^*_i | - (n-1),0 \} \leq \sum_{i=1}^n  \max\{| x_i - x^*_i | - (n-1),0 \} - 1$ and/or $\sum_{i=1}^n  \mathbbm{1}\{| x'_i - x^*_i | \leq (n-1) \} \leq \sum_{i=1}^n  \mathbbm{1}\{| x_i - x^*_i | \leq (n-1) \} - 1$, where $\mathbbm{1}(p) = 1$ if and only if $p$ is true. 
Repeating this process leads, in a finite number of steps, to a solution $\mathbf{x''}$ such that $| x''_i - x^*_i | < n-1, \text { for }  i \in \{1,\dots,n\}$.\qedhere
\end{leftbar}
\endproof

\end{document}